\documentclass[11pt]{amsart}

\usepackage{hyperref}
\usepackage{amssymb,amsfonts}
\usepackage{hyperref}
\usepackage{amssymb,textcomp}

\usepackage{enumerate}

\usepackage[utf8]{inputenc}
\usepackage[T1]{fontenc}

\usepackage[mathscr]{eucal}

\usepackage[a4paper, twoside=false, vmargin={2cm,3cm}, includehead]{geometry}

\newtheorem{lemma}{Lemma}[section]
\newtheorem{theorem}[lemma]{Theorem}
\newtheorem*{theorem*}{Theorem}
\newtheorem{corollary}[lemma]{Corollary}

\newtheorem{proposition}[lemma]{Proposition}
\newtheorem*{proposition*}{Proposition}

\newtheorem*{problem*}{Problem}

\theoremstyle{definition}
\newtheorem*{claim*}{Claim}

\newtheorem*{definition}{Definition}

\newtheorem*{remark}{Remark}
\newtheorem*{remarks}{Remarks}

\newcommand{\fdeg}{\text{f-deg}}

\newcommand{\C}{{\mathbb C}}
\newcommand{\E}{{\mathbb E}}

\newcommand{\N}{{\mathbb N}}
\renewcommand{\P}{{\mathbb P}}

\newcommand{\R}{{\mathbb R}}

\newcommand{\Z}{{\mathbb Z}}

\newcommand{\CX}{{\mathcal X}}

 \newcommand{\uh}{{\underline{h}}}

\newcommand{\norm}[1]{\left\Vert #1\right\Vert}
\newcommand{\nnorm}[1]{\lvert\!|\!| #1|\!|\!\rvert}

\begin{document}

		\title{Joint ergodicity of fractional powers of primes}

	\thanks{The author was supported  by the Research Grant - ELIDEK HFRI-FM17-1684.}

	\author{Nikos Frantzikinakis}
	\address[Nikos Frantzikinakis]{University of Crete, Department of mathematics and applied mathematics, Voutes University Campus, Heraklion 71003, Greece} \email{frantzikinakis@gmail.com}
	\begin{abstract}
		We establish mean convergence for multiple ergodic averages with iterates given by distinct
		fractional powers of primes and related multiple recurrence results.  A consequence of our main result is that every set of integers with positive  upper density
		contains patterns of the form $\{m,m+[p_n^a], m+[p_n^b]\}$, where $a,b$ are positive non-integers and $p_n$ denotes the $n$-th prime, a property that fails if $a$ or $b$ is a natural number. Our approach is based on a recent criterion for joint ergodicity  of collections of sequences and the bulk of the proof is devoted to obtaining good seminorm estimates  for the related multiple ergodic averages.  The input needed from number theory are upper bounds for the number of prime $k$-tuples that follow from
		elementary sieve theory estimates and equidistribution results of fractional powers of primes in the circle.
	\end{abstract}


\maketitle

\section{Introduction and main results}\label{S:MainResults}
\subsection{Introduction} Given an  ergodic measure preserving system $(X, \mu,T)$ and functions $f,g\in L^\infty(\mu)$,
it was shown in \cite{F10} that for distinct $a,b\in \R_+\setminus\Z$  we have
\begin{equation}\label{E:fracinteger}
\lim_{N\to\infty} \frac{1}{N}\sum_{n=1}^N \,  T^{[n^a]}f\cdot T^{[n^b]}g=\int f\, d\mu \cdot \int g\, d\mu
\end{equation}
in $L^2(\mu)$.\footnote{Throughout, with $(X,\mu,T)$ we mean a probability space $(X,\mathcal{X},\mu)$ together with an invertible, measurable,  and measure-preserving $T\colon X\to X$. The system is ergodic if the only $T$-invariant sets in $\CX$ have measure $0$ or $1$. If $f\in L^\infty(\mu)$, with $T^nf$ we denote the composition $f\circ T^n$, where $T^n:=T\circ\cdots\circ T$.}
An immediate consequence of this limit formula is that for every
(not necessarily ergodic) measure preserving system   and measurable set $A$,  we have
\begin{equation}\label{E:lower}
\lim_{N\to\infty} \frac{1}{N}\sum_{n=1}^N \, \mu(A\cap  T^{-[n^a]}A\cap T^{-[n^b]}A)\geq \mu(A)^3.
\end{equation}
Examples of periodic systems show that  \eqref{E:fracinteger} and \eqref{E:lower}
fail if either $a$ or $b$ is an integer greater than $1$.
Using  the Furstenberg correspondence principle \cite{Fu77, Fu81a}, it is easy to deduce from \eqref{E:lower}
that every set of integers
 with positive upper density contains patterns of the form
 $$
 \{m,m+[n^a],m+[n^b]\}
 $$ for some $m,n\in\N$.

The main goal of this article is to establish similar  convergence and  multiple recurrence results, and deduce related combinatorial consequences,  when in the previous statements we replace the variable  $n$ with the $n$-th prime number $p_n$. For instance, we  show in Theorem~\ref{T:Main} that if $a,b\in \R_+$ are distinct non-integers, then
\begin{equation}\label{E:fracprime}
\lim_{N\to\infty} \frac{1}{N}\sum_{n=1}^N \, T^{[p_n^a]}f\cdot T^{[p_n^b]}g=\int f\, d\mu \cdot \int g\, d\mu
\end{equation}
in $L^2(\mu)$. We also prove more general statements of this sort involving two or more linearly independent polynomials with fractional exponents evaluated at  primes (related results  for fractional powers of   integers were previously established in \cite{BMR21,F10, R21}).

If $a,b\in \N$ are natural numbers, then   \eqref{E:fracprime} fails
 because of  obvious congruence
obstructions.  On the other hand, using the method in \cite{FHK11} it can be shown that if $a,b\in\N$ are distinct, then  \eqref{E:fracprime} does hold under the additional assumption that the system is totally ergodic, see also \cite{KK19,K15} for related work regarding polynomials in $\R[t]$ evaluated at  primes.
The main  idea in the proof of these results  is to show that the difference of
 a modification of the averages in \eqref{E:fracprime} and the averages \eqref{E:fracinteger}  converges   to $0$ in $L^2(\mu)$.
 This comparison method works well when $a,b$ are positive integers, since in this case    one can bound this difference by the
 Gowers uniformity norm of the modified von Mangoldt function $\tilde{\Lambda}_N$ (see~\cite[Lemma~3.5]{FHK11} for the precise statement),
 which is known by  \cite{GT09} to converge to $0$ as $N\to\infty$.
  Unfortunately, if $a,b$ are not integers  this comparison step breaks down, since
  it requires a uniformity property for $\tilde{\Lambda}_N$  in which some
  of the averaging parameters lie in  very short intervals, a property   that is currently not known.
An alternative approach for establishing \eqref{E:fracprime} is given by the argument used in \cite{F10} to prove \eqref{E:fracinteger}. It
 uses the theory of characteristic factors that originates from \cite{HK05}
and eventually reduces the problem to an equidistribution result on nilmanifolds. This method is also blocked, since  we are unable to establish the needed equidistribution properties on general nilmanifolds.\footnote{For polynomials with  integer degrees the needed equidistribution property can be verified  using a comparison method that again breaks down when the degrees are fractional.}

Our approach is quite different and is based on a recent result of the author from \cite{F21}  (see Theorem~\ref{T:JointErgodicity} below);
 it implies that in order to verify \eqref{E:fracprime}, it suffices to  obtain suitable seminorm estimates and  equidistribution results on the circle (versus the general nilmanifold that the method of characteristic factors requires).
   The needed equidistribution property follows from  \cite{BKMST14} (see Theorem~\ref{T:BKMST} below) and
 the bulk of this article is devoted to the rather tricky  proof of the seminorm estimates  (see Theorem~\ref{T:CFprimes} below).


 \subsection{Main results}

 To facilitate  discussion we use the following  definition
 from \cite{F21}.
 	\begin{definition}
 		 We say that  the collection of sequences  $b_1,\ldots, b_\ell\colon \N\to \Z$ is  {\em jointly ergodic}, if  for every ergodic system $(X,\mu,T)$ and  functions   $f_1,\ldots, f_\ell \in L^\infty(\mu)$ we have
 		$$
 	\lim_{N\to\infty} \frac{1}{N}\sum_{n=1}^N \, T^{b_1(n)}f_1 \cdot\ldots \cdot  T^{b_\ell(n)}f_\ell= \int f_1\, d\mu\cdot \ldots \cdot \int f_\ell\, d\mu
 		$$
 		in $L^2(\mu)$.
 	\end{definition}
 For instance, the  identities   \eqref{E:fracinteger} and \eqref{E:fracprime} are equivalent to the joint ergodicity of  the pairs of sequences
 $\{[n^a], [n^b]\}$ and  $\{[p_n^a], [p_n^b]\}$  when $a,b\in \R_+$ are distinct non-integers.

 We are going to  establish joint ergodicity properties involving the  class
 of fractional polynomials that we define next.
 \begin{definition}
 	A  {\em polynomial with real exponents}   is a function $a\colon \R_+\to \R$ of the form $a(t)=\sum_{j=1}^r \alpha_jt^{d_j}$, where $\alpha_j\in\R$ and $d_j\in \R_+$, $j=1,\ldots, r$.  If  $d_1,\ldots, d_r\in \R_+\setminus \Z$,  we call it a  \emph{fractional polynomial}. 
 \end{definition}

The following is the main result of this article:

\begin{theorem}\label{T:Main}
	Let $a_1,\ldots, a_\ell$  be linearly independent\footnote{Henceforth, when we say  ``linearly independent'' we mean linearly independent over $\R$.} fractional polynomials.
	Then
	 the collection of sequences $[a_1(p_n)],\ldots, [a_\ell(p_n)]$ is jointly ergodic.
\end{theorem}
In particular, this applies to the collection of sequences $[n^{c_1}],\ldots, [n^{c_\ell}]$ where $c_1,\ldots, c_\ell\in \R_+\setminus\Z$ are distinct.
We remark also that the linear independence assumption is necessary for joint ergodicity. Indeed, suppose that $a_1,\ldots,a_\ell$ is a collection of linearly depended sequences. Then   $c_1a_1+\cdots+c_\ell a_\ell=0$ for some $c_1,\ldots, c_\ell\in \R$ not all of them $0$. After multiplying by an appropriate constant  we can assume that at least one of the $c_1,\ldots, c_\ell$ is not an integer and   $\max_{i=1,\ldots, \ell}|c_i|\leq 1/(10\ell)$. Then 
 	 $c_1[a_1(n)]+\cdots+c_\ell[a_\ell(n)]\in [-1/10,1/10]$ for all $n\in \N$, and  this  easily implies that  the collection $[a_1(n)],\ldots, [a_\ell(n)]$ is  not good for equidistribution (see definition in Section~\ref{S:strategy}) and hence not jointly ergodic.

Using standard methods we immediately deduce from Theorem~\ref{T:Main} the following
 multiple recurrence result:
\begin{corollary}
	Let $a_1,\ldots, a_\ell$  be linearly independent fractional polynomials. Then for every system $(X,\mu,T)$ and measurable set $A$ we have
	$$
\lim_{N\to\infty} \frac{1}{N}\sum_{n=1}^N \, \mu(A\cap  T^{-[a_1(p_n)]}A\cap\cdots\cap  T^{-[a_{\ell}(p_n)]}A)\geq (\mu(A))^{\ell+1}.
	$$
\end{corollary}
Using the Furstenberg correspondence principle  \cite{Fu77, Fu81a}, we deduce the following  combinatorial consequence:
\begin{corollary}
	Let  $a_1,\ldots, a_\ell$  be linearly independent fractional polynomials. Then for every subset $\Lambda$  of $\N$ we have\footnote{For $A\subset \N$ we let $\overline{d}(A):=\limsup_{N\to\infty}\frac{|A\cap [N]|}{N}$.}
	$$
	\liminf_{N\to\infty} \frac{1}{N}\sum_{n=1}^N \, \overline{d}(\Lambda\cap (\Lambda -[a_1(p_n)])\cap \cdots \cap (\Lambda -[a_\ell(p_n)]))\geq (\overline{d}(\Lambda))^{\ell+1}.
	$$
	Hence, every set of integers with positive upper density contains patterns of the form $\{m,m+[a_1(p_n)], \ldots, m+[a_\ell(p_n)] \}$ for some $m,n\in\N$.
\end{corollary}
 An essential tool  in the proof of our main result is the following statement  that is of independent interest since it covers a larger class
  of collections of    fractional polynomials (not necessarily linearly independent) evaluated at  primes. See Section~\ref{S:strategy} for the definition of the seminorms $\nnorm{\cdot}_s$.
  \begin{theorem}\label{T:CFprimes}
  	Suppose that the fractional polynomials $a_1,\ldots, a_\ell$  and their pairwise differences are non-zero. Then there exists $s\in \N$ such that  for every ergodic system  $(X,\mu,T)$ and functions $f_1,\ldots, f_\ell\in L^\infty(\mu)$ with  $\nnorm{f_i}_{s}=0$ for some $i\in\{1,\ldots, \ell\}$, we have
  	\begin{equation}\label{E:apn}
  	\lim_{N\to\infty}		\frac{1}{N}\sum_{n=1}^N\,  T^{[a_1(p_n)]}f_1\cdot \ldots \cdot T^{[a_\ell(p_n)]}f_\ell=0
  	\end{equation}
  	in $L^2(\mu)$. 
  \end{theorem}
\begin{remark}
It seems likely that 	with some additional effort
 the techniques of this article
	can  cover the more general case of Hardy field functions $a_1,\ldots, a_\ell$ such that the functions and their differences
	belong to the set $\{a\colon \R_+\to \R\colon t^{k+\varepsilon}\prec a(t)\prec t^{k+1-\varepsilon} \text{ for some } k\in \Z_+ \text{ and } \varepsilon>0\}$. Using the equidistribution result in \cite{BKS19} and the argument in Section~\ref{S:strategy}, this   would immediately  give  a corresponding strengthening of Theorem~\ref{T:Main}. We opted not to deal with these more general statements
	because the added technical complexity would obscure the main ideas of the proof of Theorem~\ref{T:CFprimes}.
	\end{remark}
 The proof of Theorem~\ref{T:CFprimes} crucially uses the fact that the iterates $a_1,\ldots, a_\ell$
 have  ``fractional power growth'' and our argument fails for iterates with ``integer power growth''.
 Similar results that cover the case of   polynomials with integer or real coefficients were obtained in \cite{FHK11, WZ11} and \cite{KK19}  respectively, and depend on deep properties of  the von Mangoldt function from \cite{GT08} and \cite{GT09},   but these results and their proofs do not appear to be useful for our purposes. Instead, we rely on some softer  number theory input that follows from standard sieve theory techniques (see Section~\ref{SS:sieve}), and an argument that is fine-tuned for the case of fractional polynomials (but fails for polynomials with integer exponents). This argument eventually enables us to bound the averages in \eqref{E:apn} with averages involving iterates given by multivariate polynomials with real coefficients evaluated at the integers, a case that was essentially handled in \cite{L05}.

\subsection{Limitations of our techniques and open problems}
We expect that  the following generalisation of Theorem~\ref{T:Main} holds:
\begin{problem*}
Let  $a_1,\ldots,a_\ell$ be functions from a Hardy field with polynomial growth such that every non-trivial linear combination $b$ of them satisfies
$
|b(t)-p(t)|/\log{t}\to \infty$ for all $p\in \Z[t]$. Then
the collection of sequences $[a_1(p_n)],\ldots, [a_\ell(p_n)]$ is jointly ergodic.
\end{problem*}
By Theorem~\ref{T:JointErgodicity}, it suffices to show that the collection $[a_1(p_n)],\ldots, [a_\ell(p_n)]$ is good for equidistribution and seminorm estimates.
Although the needed  equidistribution property  has been proved in \cite[Theorem~3.1]{BKS19}, the  seminorm estimates that extend Theorem~\ref{T:CFprimes}  seem  hard to establish. Our argument breaks down when some of the functions, or their differences, are  close to integral powers of $t$, for example when they are  $t^k\log{t}$ or $t^k/\log\log{t}$ for some $k\in \N$.
  In both cases the vdC-operation (see Section~\ref{SS:VDC}) leads   to sequences of sublinear growth  for which we can no longer establish  Lemma~\ref{L:11}, in the first case because the estimate \eqref{E:654'} fails and in the second case because   in \eqref{E:z}   the length of the interval in the average is too short for Corollary~\ref{C:primes}
  to be applicable.

Finally, we remark that although the reduction offered by Theorem~\ref{T:JointErgodicity}   is very helpful when  dealing with  averages with independent iterates, as is the case in \eqref{E:fracprime},
it does not offer any help when the iterates are linearly dependent, which is the case for the averages 																			
\begin{equation}\label{E:AP}
\frac{1}{N}\sum_{n=1}^N \, T^{[p_n^{a}]}f\cdot T^{2[p_n^a]}g,
\end{equation}
where $a\in \R_+$ is not an integer.
We  do expect  the   $L^2(\mu)$-limit of the averages \eqref{E:AP} to be  equal to the $L^2(\mu)$-limit of the averages $ \frac{1}{N}\sum_{n=1}^N \,T^{n}f\cdot T^{2n}g$, but this remains a challenging open problem\footnote{Although the method of Theorem~\ref{T:CF} does give good seminorm bounds for the averages \eqref{E:AP}, the needed equidistribution properties on nilmanifolds present serious difficulties.}, see Problem~27 in \cite{Fr16}.





\subsection{Notation}
With $\N$ we denote the set of positive integers and with $\Z_+$ the set of non-negative integers. With $\P$ we denote the set of prime numbers.
With  $\R_+$ we denote the set of non-negative real numbers.
For $t\in \R$ we let $e(t):=e^{2\pi i t}$.
If $x\in \R_+$, 	when there is no danger for confusion, with   $[x]$ we denote both the integer part of $x$ and the set $\{1,\ldots, [x]\}$.
We denote with $\Re(z)$ the real part of the complex number $z$.

Let $a\colon \N\to \C$ be a  bounded sequence.
   If $A$ is a non-empty finite subset of $\N$  we let
$$
\E_{n\in A}\,a(n):=\frac{1}{|A|}\sum_{n\in A}\, a(n).
$$

If $a,b\colon \R_+\to \R$ are functions we write
\begin{itemize}
	\item  $a(t)\prec b(t)$ if $\lim_{t\to +\infty} a(t)/b(t)=0$;
	
	
	\item   $a(t)\sim b(t)$ if $\lim_{t\to +\infty} a(t)/b(t)$ exists and is non-zero;
	
	\item  $A_{c_1,\ldots, c_\ell}(t)\ll_{c_1,\ldots, c_\ell} B_{c_1,\ldots, c_\ell}(t)$ if there exist $t_0=t_0(c_1,\ldots, c_\ell)\in \R_+$ and  $C=C(c_1,\ldots, c_\ell)>0$ such that $|A_{c_1,\ldots, c_\ell}(t)|\leq  C |B_{c_1,\ldots, c_\ell}(t)|$ for all
	$t\geq t_0$.
\end{itemize}
We use the same notation for sequences $a,b\colon \N\to \R$.

Throughout, we let $L_N:=[e^{\sqrt{\log{N}}}]$, $N\in\N$.

We say that a sequence $(c_{N,\uh}(n))$ where $\uh\in [L_N]^k$, $n\in [N]$, $N\in\N$, is bounded, if there exists C>0 such that $|c_{N,\uh}(n)|\leq C$ for all
$\uh\in [L_N]^k$, $n\in [N]$, $N\in\N$.


\subsection{Acknowledgement} The author would like to thank the two anonymous referees for their valuable comments.

\section{Proof strategy}\label{S:strategy}
Our argument depends upon  a convenient criterion for joint ergodicity that was established recently in \cite{F21} (and was motivated by work in \cite{P19, PP19}).
In order to state it we  need to review the definition of the ergodic seminorms from \cite{HK05}.

\begin{definition}
For a given   ergodic system
$(X,\mu,T)$ and function $f\in L^\infty(\mu)$, we define $\nnorm{\cdot}_s$ inductively
as follows:
\begin{gather*}
	\label{E:first}
	\nnorm{f}_{1}\mathrel{\mathop:}=\Big| \int f \ d\mu\Big|\ ;\\
	\label{E:recur}
	\nnorm f_{s+1}^{2^{s+1}} \mathrel{\mathop:}=\lim_{N\to\infty}\frac{1}{N}\sum_{n=1}^{N}
	\nnorm{\bar{f}\cdot T^nf}_{s}^{2^{s}}, \quad s\in \N.
\end{gather*}
\end{definition}
It was shown in \cite{HK05}, via successive uses of the mean ergodic theorem, that for every $s\in \N$  the  above limit
exists and $\nnorm{\cdot}_s$ defines an increasing sequence of seminorms on
$L^\infty(\mu)$.
\begin{definition}
	We say that  the collection of sequences  $b_1,\ldots, b_\ell\colon \N\to \Z$ is:
	\begin{enumerate}
		\item {\em   good for seminorm estimates,} if  for every ergodic system $(X,\mu,T)$  there exists $s\in \N$ such that if $f_1,\ldots, f_\ell \in L^\infty(\mu)$
		and $\nnorm{f_m}_{s}=0$ for some $m\in \{1,\ldots, \ell\}$, then
		\begin{equation}\label{E:ml}
		\lim_{N\to\infty} \E_{n\in [N]}\, T^{b_1(n)}f_1\cdot\ldots \cdot  T^{b_m(n)}f_m= 0
		\end{equation}
		 in $L^2(\mu)$.\footnote{In practice,  $s$ can often be chosen independently of the system and   \eqref{E:ml} can be established with $m=\ell$ (note that Property~(i) with $m=\ell$ in \eqref{E:ml}  is a stronger property than  Property~(i) as stated).}
		
		\item {\em good for equidistribution}, if for all  $t_1,\ldots, t_\ell\in [0,1)$,  not all of them $0$, we have
		$$
		\lim_{N\to\infty} \E_{n\in[N]}\,  e(b_1(n)t_1+\cdots+ b_\ell(n)t_\ell) =0.
		$$
\end{enumerate}
\end{definition}
We remark that any collection of non-constant integer polynomial sequences with pairwise non-constant differences is known to be good for seminorm estimates \cite{L05} and examples of periodic systems show that no such collection is good for equidistribution (unless $\ell=1$ and $b_1(t)=\pm t+k$). On the other hand,
a collection of linearly independent fractional polynomials is known to be good both for seminorm estimates \cite[Theorem~2.9]{F10} and equidistribution (follows from \cite[Theorem~3.4]{KN74} and  \cite[Lemma~6.2]{F21}).

A crucial ingredient used  in the proof of our main result is the following result that gives convenient necessary and sufficient conditions for  joint ergodicity of a collection of sequences (see also \cite{BM21} for an extension of this result for sequences $b_1,\ldots, b_\ell\colon \N^k\to \Z$).
\begin{theorem}[\cite{F21}]\label{T:JointErgodicity}
	The sequences    $b_1,\ldots, b_\ell\colon \N\to \Z$  are jointly ergodic if and only if they are good for equidistribution and seminorm estimates.
\end{theorem}
\begin{remark}
The 	proof of this result  uses ``soft'' tools from ergodic theory,
and avoids  deeper tools like the Host-Kra theory of characteristic factors (see \cite[Chapter~21]{HK18} for a detailed description) and equidistribution results on nilmanifolds.
	\end{remark}
In view of this result, in order to establish Theorem~\ref{T:Main}, it suffices to show that a collection of linearly independent fractional polynomials evaluated at  primes  is  good for seminorm estimates and equidistribution.
The good equidistribution property is a consequence of the following  result \cite[Theorem~2.1]{BKMST14}:
\begin{theorem}[\cite{BKMST14}]\label{T:BKMST}
	If $a(t)$ is a non-zero fractional polynomial, then the sequence
	$(a(p_n))$ is equidistributed $\!  \!  \pmod{1}$.
\end{theorem}
Using the previous result and \cite[Lemma~6.2]{F21} we immediately deduce the following:
\begin{corollary}\label{C:equi}
If $a_1,\ldots, a_\ell$ are linearly independent fractional polynomials, then
the collection of sequences $[a_1(p_n)], \ldots, [a_\ell(p_n)]$ is good for equidistribution.
\end{corollary}
We let $\Lambda'\colon \N\to \R_+$ be the following  slight modification of the  von Mangoldt function: $\Lambda'(n):=\log(n)$ if $n$ is a  prime number  and $0$ otherwise.
To establish that the collection  $[a_1(p_n)], \ldots, [a_\ell(p_n)]$ is good
for  seminorm estimates it suffices to prove the following result (the case  $w_N(n):=\Lambda'(n)$, $N,n\in \N$,   implies Theorem~\ref{T:CFprimes} in a standard way, see for example \cite[Lemma~2.1]{FHK11}):

\begin{theorem}\label{T:CF}
	Suppose that the fractional polynomials
	$a_1,\ldots, a_\ell$ and their pairwise differences are  non-zero. Then there exists $s\in \N$ such that the following holds: If $(X,\mu,T)$ is an ergodic system and $f_1,\ldots, f_\ell\in L^\infty(\mu)$ are such that  $\nnorm{f_i}_{s}=0$ for some $i\in \{1,\ldots, \ell\}$, then for every $1$-bounded sequence $(c_{N}(n))$ we have
	\begin{equation}\label{E:i1}
	\lim_{N\to\infty}		\E_{n\in [N]}\,  w_N(n) \cdot T^{[a_1(n)]}f_1\cdot \ldots \cdot T^{[a_\ell(n)]}f_\ell=0
	\end{equation}
	 in $L^2(\mu)$,  where $w_N(n):=\Lambda'(n)\cdot c_N(n)$, $n\in[N]$, $N\in\N$. 	
\end{theorem}
\begin{remarks}
	$\bullet$
The  sequence $(c_N(n))$ is not essential in order to deduce Theorem~\ref{T:CFprimes}.
It is used because it helps us  absorb 
error term that often appear in our argument.

$\bullet$	Our proof shows that
the place of the sequence $(\Lambda'(n))$ can take any non-negative sequence
$(b(n))$ that satisfies  properties $(i)$ and $(ii)$ of Corollary~\ref{C:primes}   and the estimate $b(n)\ll n^\varepsilon$ for every $\varepsilon>0$.
\end{remarks}
In order to prove Theorem~\ref{T:CF}, we use an induction argument, similar to the
polynomial exhaustion technique (PET-induction) introduced in \cite{B87}, which is based on variants of the van der Corput inequality stated immediately after  Lemma~\ref{L:VDC}. The fact that the weight sequence $(w_N(n))$ is unbounded, forces us to apply Lemma~\ref{L:VDC} in the form given in \eqref{E:VDC1}  with $L_N\in \N$ that satisfy   $L_N\succ (\log{N})^A$ for every $A>0$.  On the other hand, since we have to take care of some error terms  that are of the order $L_N^B/N^a$ for
arbitrary $a, B>0$, we are also  forced to take   $L_N\prec N^a$ for every $a>0$ in order for these errors to be negligible. These two estimates are satisfied for example when
$L_N=[e^{\sqrt{\log{N}}}]$, $N\in\N$, which is the value of $L_N$ that we use henceforth.

During the course of the PET-induction argument, we have to keep close track of the additional parameters $h_1,..., h_k$ that arise after each   application of  Lemma~\ref{L:VDC} in the form that is given  in \eqref{E:VDC1}. This is the reason why we prove
a more general variant of Theorem~\ref{T:CF} that is stated in Theorem~\ref{T:CF'}  and involves fractional polynomials with coefficients depending on finitely many parameters. It turns out that the most laborious part of its proof is the base case of the induction   where all iterates have sublinear growth. This case
is dealt in three steps. First, in Lemma~\ref{L:11} we use a change of variables argument and the number theory input from Corollary~\ref{C:primes} in order to reduce matters to the case where the weight sequence $(w_N(n))$ is bounded.
Next, in Lemma~\ref{L:12} we use  another change of variables argument
and Lemma~\ref{L:VDC}
in order to successively ``eliminate'' the sequences $a_1, \ldots, a_\ell$, and, after
$\ell$-iterations, we get an upper bound that involves iterates given by the integer parts of polynomials in several variables   with real coefficients. Lastly, in Lemma~\ref{L:13} we show that averages with such iterates obey good seminorm bounds. This
last step is carried out by adapting an argument from \cite{L05} to our setup; this is done by another PET-induction, which this time uses Lemma~\ref{L:VDC} in the form that is given  in \eqref{E:VDC2}. In Sections~\ref{Ex1} and \ref{Ex2} the reader will find  examples that explain how these arguments work in some model cases that contain the essential ideas of the general arguments.

To conclude  this  section, we remark that  in order to prove Theorem~\ref{T:Main} it suffices to prove Theorem~\ref{T:CF}; the remaining sections are devoted to this task.

\section{Seminorm estimates - Some preparation}
\subsection{A more general statement}
In  order to prove Theorem~\ref{T:CF}, it will be  convenient to establish a technically more complicated statement that is better suited for a PET-induction argument. We state it in this subsection.



 Throughout, the sequence  $L_N$  is chosen to satisfy $(\log{N})^A\prec L_N\prec N^a$ for all $A,a>0$; so we can take for example
 $$
 L_N:=[e^{\sqrt{\log{N}}}], \quad N\in \N.
 $$

 	

With $\R[t_1,\ldots, t_k]$ we denote  the set of polynomials with real coefficients in $k$-variables.
\begin{definition}
	We say that $a\colon \Z^k\times \R_+\to \R$ is a {\em polynomial with real exponents and $k$-parameters}, if it has the form
	$$
	a(\uh,t)=\sum_{j=0}^r  p_j(\uh)\, t^{d_j},
	$$
	for some $r\in \N$, $0=d_0<d_1<\cdots<d_r\in \R_+$, and
	$p_0,\ldots, p_r\in \R[t_1,\ldots, t_k]$. If $d_1,\ldots, d_r\in \R_+\setminus \Z$, we call it a {\em fractional polynomial with $k$-parameters}.
	 If $p_j$ is non-zero for some $j\in \{1,\ldots, r\},$ we say that  $a(\uh,t)$ is {\em non-constant}.
	We define the  {\em fractional degree} of $a(\uh,t)$, and denote it by  $\fdeg(a)$, to be the maximum exponent $d_j$ for which the polynomial $p_j$ is non-zero.  We call the integer part of its fractional degree  the {\em  (integral) degree}  of $a(\uh,t)$ and denote it by $\deg(a)$. We also let $\deg(0):=-1$.
	\end{definition}
 For example, the fractional polynomial with $1$-parameter
$h^2 t^{0.5}+(h^2\sqrt{2}+h)t^{0.1}$ has fractional degree $0.5$ and degree $0$.

\begin{definition}
 We say that a collection $a_1,\ldots, a_\ell$ of polynomials with real exponents  and $k$-parameters is {\em nice} if
 \begin{enumerate}
 \item  $\fdeg(a_i)\leq \fdeg(a_1)$ for $i=2,\ldots, \ell$, and

 \item    the functions $a_1, \ldots a_\ell$ and the functions  $a_1-a_2, \ldots, a_1-a_\ell$   are non-constant in the variable $t$ (and as a consequence they have positive fractional degree).
 \end{enumerate}
\end{definition}



Given a sequence $u\colon\N\to \C$, we let $(\Delta_hu)(n):=u(n+h)\cdot \overline{u(n)}$, $h,n\in\N$, and if $\uh=(h_1,\ldots, h_k)$, we let $(\Delta_{\uh})(u(n)):=(\Delta_{h_k}\cdots\Delta_{h_1}) (u(n))$, $h_1,\ldots, h_k,n\in\N$. For example, $(\Delta_{(h_1,h_2)})(u(n))= u(n+h_1+h_2)\cdot \overline{u}(n+h_1)\cdot \overline{u}(n+h_2)\cdot u(n)$, $h_1,h_2,n\in\N$.
\begin{theorem}\label{T:CF'}
	For $k\in\Z_+, \ell\in\N$, let    $a_1,\ldots, a_\ell\colon \N^k\times \N\to \R$ be  a nice collection of fractional  polynomials with $k$-parameters and $(c_{N,\uh}(n))$ be a $1$-bounded sequence. Then
	there exists $s\in \N$ such that the following holds: If $(X,\mu,T)$ is a system and  $f_{N,\uh,1},\ldots, f_{N,\uh,\ell}\in L^\infty(\mu)$, $\uh\in [L_N]^k,N\in\N$,  are $1$-bounded functions with  $f_{N,\uh,1}=f_1$, $\uh\in \N^k,N\in\N$, and
	$\nnorm{f_1}_s=0$, then
	\begin{equation}\label{E:mainav}
	\lim_{N\to\infty}	\E_{\underline{h}\in [L_N]^k}	
	\norm{\E_{n\in [N]}\,  w_{N,\underline{h}}(n)\cdot \prod_{i=1}^\ell T^{[a_i(\underline{h},n)]}f_{N,\uh,i}}_{L^2(\mu)}=0,
	\end{equation}
	where  $w_{N,\uh}(n):=(\Delta_\uh\Lambda')(n)\cdot c_{N,\uh}(n)$, $ \uh\in [L_N]^k, n\in [N],  N\in\N$.
\end{theorem}
\begin{remark}
	Our argument also works if   $\Delta_\uh\Lambda'(n)$ is replaced by other expressions involving $\Lambda'$, for example when  $k=0$  one can use the expression   $\prod_{i=1}^m\Lambda'(n+c_i)$, where $c_1,\ldots, c_m$ are distinct integers.   
	\end{remark}
If in Theorem~\ref{T:CF'}  we take $k=0$,  then we get Theorem~\ref{T:CF} using an argument that we describe next.
 \begin{proof}[Proof of Theorem~\ref{T:CF} assuming Theorem~\ref{T:CF'}]
 	Let
 	$a_1,\ldots, a_\ell$ and $w_N(n)$ be as in  Theorem~\ref{T:CF}. Since the assumptions of Theorem~\ref{T:CF} are symmetric with respect to the sequences $a_1,\ldots, a_\ell$ it suffices to show that there exists $s\in\N$ such  that if  $\nnorm{f_1}_s=0$, then \eqref{E:i1} holds. 
 	
 	If $a_1$ has maximal  fractional degree within the family $a_1,\ldots a_\ell$, then if we take $k=0$ and all functions to be independent of $N$ in Theorem~\ref{T:CF'},  we get that the conclusion of  Theorem~\ref{T:CF} holds.  Otherwise, we can assume that $a_{\ell}$ is the function with the highest fractional degree and, as a consequence,  $\fdeg(a_1)<\fdeg(a_\ell)$.
 	It suffices to show that 
 	$$
 	\lim_{N\to\infty}
 	\E_{n\in [N]}\,  w_{N}(n)\cdot \int   f_{N,0} \cdot  \prod_{i=1}^\ell T^{[a_i(n)]}f_{i} \, d\mu=0,
 	$$
 	where
 	$$
 	f_{N,0} :=    \E_{n\in [N]} \, \overline{w}_{N}(n)\cdot    \prod_{i=1}^\ell T^{[a_i(n)]}\overline{f}_i, \quad   N\in \N.
 	$$
 	Note that since $f_1,\ldots, f_\ell$ and $c_N$ are $1$-bounded we have
 	$$
 	\limsup_{N\to\infty}	\norm{f_{N,0}}_\infty\leq   \lim_{N\to\infty} \E_{n\in [N]} \, \Lambda'(n)=1,
 	$$
 	(the last identity follows from the prime number theorem but we only need the much simpler upper bound) hence, we can assume that $f_{N,0}$ is $1$-bounded for every $N\in \N$.

 	After  composing with $T^{-[a_{\ell}(n)]}$, using the Cauchy-Schwarz inequality, and the identity $[x]-[y]=[x-y]+e$ for some $e\in \{0,1\}$,   we are reduced to showing that 
 		$$
 	\lim_{N\to\infty}
 	\norm{\E_{n\in [N]}\,  w_{N}(n)\cdot     \prod_{i=1}^{\ell-1} T^{[a_i(n)-a_\ell(n)]+e_i(n)}f_i\cdot  T^{[-a_{\ell}(n)]+e_\ell(n)}f_{N,0}}_{L^2(\mu)}=0,
 	$$
 	for some $e_1(n),\ldots, e_{\ell-1}(n)\in\{0,1\}$, $n\in\N$. 
 	Next, we would like to replace the error sequences $e_1,\ldots, e_{\ell-1}$ with constant sequences. To this end,    we	 use Lemma~\ref{L:errors} for $I$ a singleton, $J:=[N]$, $X:=L^\infty(\mu)$,  $A_N(n_1,\ldots,n_\ell):=\prod_{i=1}^{\ell-1} T^{n_i}f_i\cdot  T^{-n_\ell}f_{N,0}$, $n_1,\ldots, n_\ell\in \Z$, and 
 		$b_i:=[a_i-a_{\ell}]$, $i=1,\ldots, \ell-1$, $b_\ell:= [-a_\ell]$.
 		  We get that it suffices to show that
 	\begin{equation}\label{E:zN}
 	\lim_{N\to\infty}	
 	\norm{\E_{n\in [N]}\,  z_{N}(n)\cdot \prod_{i=1}^\ell T^{[a'_i(n)]}g_{N,i}}_{L^2(\mu)}=0,
 	\end{equation}
 where 
 		$$
 	a'_i:=a_i-a_{\ell}, \quad i=1,\ldots, \ell-1, \quad a'_\ell:= -a_\ell.
 	$$
 	for some $1$-bounded sequence $(z_{N}(n))$, where
 $g_{N,i}:=T^{\epsilon_i}f_{i}$, $i=1,\ldots, \ell-1$, $g_{N,\ell}:=T^{\epsilon_\ell}f_{N,0}$, $N\in \N$, for some constants $\epsilon_1,\ldots, \epsilon_\ell\in \{0,1\}$.
 	Note that the family $a'_1,\ldots, a'_\ell$ is nice, and  $g_{N,1}=T^{\epsilon_1}f_1$,  $N\in \N$, so Theorem~\ref{T:CF'} applies (for $k=0$ and all but one of the functions independent of $N$)  and gives that there exists $s\in \N$ so that if $\nnorm{f_1}_s=0$, then \eqref{E:zN} holds. This completes the proof.
 \end{proof}

 We will prove Theorem~\ref{T:CF'} in  Sections~4 and 5 using a
 PET-induction technique. The first section covers the base case of the induction  where all the iterates have sublinear growth and the subsequent section contains the proof of the induction step. Before moving into the details we gather some basic tools that will be used in the argument.

\subsection{Feedback from number theory} \label{SS:sieve}
The next statement
is well known and  can be proved using elementary sieve theory methods (see for example \cite[Theorem~5.7]{HR74} or \cite[Theorem~6.7]{IK04}).
\begin{theorem}\label{T:primes}
	Let $\P$ be the set of prime numbers. For all $k\in \N$
	there exist $C_k>0$  such that  for	all distinct  $h_1,\ldots, h_k\in \N$ and  all $N\in \N$
	we have
$$
		|\{n\in [N]\colon n+h_1,\ldots, n+h_k \in \P\}|\leq C_k\,  \mathfrak{G}_k(h_1,\ldots, h_k)\,  \frac{N}{(\log{N})^k},
	$$
where
\begin{equation}\label{E:SS}
\mathfrak{G}_k(h_1,\ldots, h_k):=\prod_{p\in \P} \Big(1-\frac{1}{p}\Big)^{-k}\Big(1-\frac{\nu_p(h_1,\ldots, h_k)}{p}\Big)
\end{equation}
and $\nu_p(h_1,\ldots, h_k)$ denotes the number of congruence classes $\! \! \mod{p}$ that are occupied by $h_1,\ldots, h_k$.
\end{theorem}
We remark that although $\mathfrak{G}_1=1$,   the expression $\mathfrak{G}_k(h_1,\ldots, h_k)$ is not bounded in $h_1,\ldots, h_k$ if $k\geq 2$,  and this  causes some problems for us.
Asymptotics for averages of  powers of $\mathfrak{G}_k(h_1,\ldots, h_k)$ are given in \cite{G76} and \cite[Theorem~1.1]{K11} using elementary  but somewhat elaborate arguments. These results are not immediately applicable for our purposes, since we need to understand the behavior of $\mathfrak{G}_k$ on thin
subsets of $\Z^k$, for instance,  when $k=4$ we need to  understand the averages of   $\mathfrak{G}_4(0,h_1,h_2,h_1+h_2)$. Luckily,  we only need to get upper bounds for these averages and this can be done rather easily as we will see shortly (a similar argument was used in \cite{TZ15} to handle
 averages over $r$ of $\mathfrak{G}_k(0,r,2r,\ldots, (k-1)r)$).
 \begin{definition}
 	Let $\ell \in \N$ and for $\uh\in \N^\ell$ let 	$\text{Cube}(\uh)\in \N^{2^\ell}$ be defined by
 	$$
 	\text{cube}(\uh):=(\underline\epsilon\cdot \uh)_{\underline\epsilon\in\{0,1\}^\ell},
 	$$
 	where $\underline\epsilon\cdot \uh$ is the inner  product of $\underline\epsilon$ and $\uh$.
 	
If $S$ is a subset of  $\N^\ell$, 	we  define
 	$$
 	S^*:=\{\uh\in S\colon \text{cube}(\uh) \text{ has distinct coordinates}\}.
 	$$
 \end{definition} 	
 	For instance, when $\ell=3$ we have
 	$$
 	\text{cube}(h_1,h_2,h_3)=(0,h_1,h_2,h_3,h_1+h_2, h_1+h_3,h_2+h_3,h_1+h_2+h_3)
$$
and $([N]^3)^*$ consists of all triples $(h_1,h_2,h_3)\in [N]^3$ with distinct coordinates that in addition satisfy $h_i\neq h_j+h_k$ for all distinct $i,j,k\in \{1,2,3\}$. Since the complement of $([N]^\ell)^*$ in $[N]^\ell$ is contained on the zero set of finitely many (at most $3^\ell$) linear forms,  we get that there exists $K_\ell>0$ such that
\begin{equation}\label{E:card*}
 |[N]^\ell\setminus ([N]^\ell)^*|\leq K_\ell\, N^{\ell-1}
\end{equation}
for every $N\in\N$.

\begin{proposition}\label{P:SS}
For  every  $\ell \in \N$ there exists $C_\ell>0$ such that
$$
\E_{\uh \in [N]^\ell}  \big(\mathfrak{G}_{2^\ell}(\text{cube}(\uh) )\big)^2\leq C_\ell,
$$
for all $N\in \N$, where  $\mathfrak{G}_{2^\ell}(\text{cube}(\uh))$ is as in \eqref{E:SS}.
\end{proposition}
\begin{remark}
	If we use $k$-th powers instead of squares we get similar upper bounds (which also depend on $k$), but we will not need this.
	\end{remark}
\begin{proof}
	 In the following argument 	whenever we write $p$ we assume that
	  $p$ is a prime number.
	
Let $\uh\in [N]^\ell$. 	Note that if $\nu_p(\text{cube}(\uh))=2^\ell$, then
$$
\Big(1-\frac{1}{p}\Big)^{-2^\ell}\Big(1-\frac{\nu_p(\text{cube}(\uh))}{p}\Big)\leq 1,
$$
and if $\nu_p(\text{cube}(\uh))<2^\ell$,  then   for  $a_\ell:=2^{\ell+1}-2$ we have
$$
\Big(1-\frac{1}{p}\Big)^{-2^\ell}\Big(1-\frac{\nu_p(\text{cube}(\uh))}{p}\Big)\leq \Big(1-\frac{1}{p}\Big)^{-(2^\ell-1)}\leq  e^{\frac{a_\ell}{p}},
$$
where we used that  $\frac{1}{1-x}\leq e^{2x}$ for $x\in [0,\frac{1}{2}]$.
Note also that if  $\nu_p(\text{cube}(\uh))<2^\ell$, then there exist distinct  $\underline\epsilon,\underline\epsilon'\in \{0,1\}^{\ell}$
such that $p|(\underline\epsilon-\underline\epsilon')\cdot \uh$, in which case we have that  $p\in \mathcal{P}(\uh)$ where
$$
\mathcal{P}(\uh):=\bigcup_{\underline\epsilon,\underline\epsilon' \in \{0,1\}^\ell, \underline\epsilon,\neq \underline\epsilon' }\,  \{p\in \P \colon p|(\underline\epsilon-\underline\epsilon')\cdot \uh\}, \quad \uh\in \N^\ell.
$$

We deduce from the above facts and \eqref{E:SS} that
\begin{equation}\label{E:est1}
 \mathfrak{G}_{2^\ell}(\text{cube}(\uh)) \leq  e^{a_\ell \sum_{p\in \mathcal{P}(\uh)}\frac{1}{p}}.
\end{equation}

By  \cite[Lemma E.1]{TZ08}  we have for some $b_\ell,c_\ell>0$ that
\begin{equation}\label{E:est2}
 e^{a_\ell \sum_{p\in \mathcal{P}(\uh)}\frac{1}{p}}
 \leq b_\ell \sum_{p\in \mathcal{P}(\uh)}\frac{(\log{p})^{c_\ell}}{p}=
b_\ell \sum_{\underline\epsilon,\underline\epsilon' \in \{0,1\}^\ell, \underline\epsilon\neq \underline\epsilon'}\Big(\sum_{ p|(\underline\epsilon-\underline\epsilon')\cdot \uh}\frac{(\log{p})^{c_\ell}}{p}\Big).
\end{equation}
Moreover, we get for some $d_\ell,e_\ell>0$ that
\begin{equation}\label{E:est3}
\sum_{\uh\in [N]^\ell}\Big(\sum_{ p|(\underline\epsilon-\underline\epsilon')\cdot \uh}\frac{(\log{p})^{c_\ell}}{p}\Big)\leq
d_\ell \sum_{ p}\frac{(\log{p})^{c_\ell}}{p} \frac{N^\ell}{p}\leq e_\ell\,  N^\ell,
\end{equation}
 for all $N\in\N$, where to get the first estimate we used the fact that for some $d_\ell>0$ we have
 $$
 |\uh\in[N]^\ell\colon p|(\underline\epsilon-\underline\epsilon')\cdot \uh|\leq d_\ell \frac{N^\ell}{p},
$$
for all $N\in \N$, and to get the second estimate we used that
$\sum_{ p}\frac{(\log{p})^{c_\ell}}{p^2}<\infty$.

If we take squares in \eqref{E:est1} and sum over all $\uh\in [N]^\ell$, then use \eqref{E:est2} and \eqref{E:est3}, we get the asserted estimate.
\end{proof}
	
	
From this we deduce the following  estimate that is a crucial ingredient used in the proof of Theorem~\ref{T:CF}:

 \begin{corollary}\label{C:primes}
 	Let $\ell\in \N$. 		Then
 	for every  $A\geq 1$ there exist $C_{A,\ell} (\uh)>0$,  $\uh\in\N^\ell$, and $D_{A,\ell}>0$,
 	such that
 	\begin{enumerate}
 	\item for all  $N\in\N$, $\uh=(h_1,\ldots,h_\ell)\in (\N^\ell)^*, c\in \N,$  such that  $c+ h_1+\cdots +h_\ell\leq N^A$, we have
 $$
 \E_{n\in [N]}\,  (\Delta_\uh \Lambda')(n+c)\leq C_{A,\ell}(\uh);
 $$

 	\item $\E_{\uh\in [H]^\ell}  (C_{A,\ell}(\uh))^2\leq D_{A,\ell}$ for every $H\in \N$.
\end{enumerate}	
 \end{corollary}
\begin{remark}
	We will use this result  in  the proof of Lemma~\ref{L:11}  for values of $c$ that are  larger than $N$ and smaller than  $N^A$ for some $A>0$   (the choice of $A$  depends on the situation).
	\end{remark}
\begin{proof}
	Since $\Lambda'$ is supported on primes and  $c+h_1+\cdots +h_\ell\leq N^A$,  we have that
	$$
	\sum_{n\in [N]}\,   (\Delta_\uh\Lambda')(n+c)\leq 	|\{n\in [N]\colon \underline{n+c}+\text{cube}(\uh)\in \P^{2^\ell}\}|\cdot (\log(N+N^A))^{2^\ell},
	$$
	where $\underline{n+c}$ is a vector with $2^\ell$ coordinates, all equal to $n+c$.
	Note that for $\uh\in (\N^\ell)^*$ we can apply Theorem~\ref{T:primes} and  we get that  there exists $D_{A,\ell}>0$ such that for every $N\in \N$ the  last expression is bounded by
	$$
	D_{A,\ell}\,  \mathfrak{G}_{2^\ell}(\text{cube}(\uh))\, N.
	$$
	If we let $C_{A,\ell}(\uh):=D_{A,\ell}\,  \mathfrak{G}_{2^\ell}(\text{cube}(\uh))$, $\uh\in \N^\ell$,  and use
	Proposition~\ref{P:SS}, we get that properties $(i)$ and $(ii)$ hold.
	\end{proof}

\subsection{Two elementary lemmas}
We will use the following inner product space variant of a classical elementary estimate of van der Corput (see  \cite[Lemma~3.1]{KN74}):
\begin{lemma}\label{L:VDC}
 Let $N\in\N$ and $(u(n))_{n\in [N]}$ be vectors in some inner product space. Then for
all $H\in[N]$ we have
$$
\norm{\E_{n\in [N]}\,  u(n)}^2\leq \frac{2}{H}\, \E_{n\in [N]}\norm{u(n)}^2 +
	4\, \E_{h\in[H]}\Big(1-\frac{h}{H}\Big)\Re\Big(\frac{1}{N}\sum_{n=1}^{N-h}
	\langle u(n+h),u(n)\rangle \Big).
	$$
\end{lemma}
We will  apply the previous lemma in the following two  cases, depending on the range of the shift parameter $h$ (the first case will be used when the relevant sequences are not necessarily  bounded).
\begin{enumerate}
\item 	If
$M_N:=1+\max_{n\in [N]}\norm{u_N(n)}^2$, $N\in\N$,  and  $L_N$ are such that  $M_N\prec L_N\prec\frac{N}{M_N}$, then for $H:=L_N$ we have
\begin{equation}\label{E:VDC1}
\norm{\E_{n\in [N]} \, u_N(n)}^2\leq 4\,
\E_{h\in [L_N]}
\Big|\E_{n\in[N]}\langle u_N(n+h), u_N(n)\rangle \Big| +o_N(1),
\end{equation}
where  for every fixed $N\in\N$ the sequence $(u_N(n))$ is either defined on the larger interval  $[N+L_N]$ or it is extended to be zero outside the interval $[N]$. In all the cases where we will apply this estimate we have $M_N\ll (\log N)^A$ for some $A>0$ and we take $L_N=[e^{\sqrt{\log{N}} }]$, $N\in \N$.

\item
If the sequence $(u_N(n))$ is bounded, then  we have
\begin{equation}\label{E:VDC2}
\limsup_{N\to\infty}
	\norm{\E_{n\in [N]}\,  u_N(n)}^2\leq 4\,
\limsup_{H\to\infty}	\E_{h\in [H]}
	\limsup_{N\to\infty}\Big|\E_{n\in[N]}\langle u_N(n+h), u_N(n)\rangle \Big|,
		\end{equation}
		where for every fixed $N\in\N$ the sequence $(u_N(n))$ is either defined on the larger interval  $[N+H]$ or it is extended to be zero outside the interval $[N]$.

\end{enumerate}


We will also make frequent use of the following simple lemma, or variants of it,  to replace error sequences that take finitely many integer values with constant sequences.
\begin{lemma}\label{L:errors}
	For $f,\ell\in \N$ there exists $C_{f,\ell}>0$ such that the following holds: 	
	Let $(X,\norm{\cdot })$ be a normed space and $F$ be a finite subset of $\Z$ with $|F|=f$, $k\in \N$, and $I\subset \N^k$, $J\subset\N$ be finite.  For $\uh\in I$, consider
	sequences $A_\uh\colon \Z^\ell\to X$, $b_{1,\uh},\ldots, b_{\ell,\uh}\colon J\to \Z$, $w_\uh \colon J\to \C$, and $e_{1,\uh},\ldots, e_{\ell,\uh}\colon J\to F$. Then   there exist sequences $\tilde{w}_\uh\colon J\to \C$, $\uh\in I$, with $\norm{\tilde{w}_\uh}_{L^\infty(J)}\leq \norm{w_\uh}_{L^\infty(J)}$, and constants $\epsilon_1,\ldots, \epsilon_\ell\in F$, such that
	\begin{multline*}
	\sum_{\uh\in I}	\Big|\Big|\sum_{n\in J}\,  w_\uh(n)\cdot  A_{\uh}(b_{1,\uh}(n)+e_{1,\uh}(n),\ldots, b_{\ell,\uh}(n)+e_{\ell,\uh}(n))	\Big|\Big|
		\leq \\
		C_{f,\ell}\, \sum_{\uh\in I}\Big|\Big|\sum_{n\in J} \,\tilde{w}_\uh(n)\cdot  A_\uh(b_{1,\uh}(n)+\epsilon_1,\ldots, b_{\ell,\uh}(n)+\epsilon_\ell)\Big|\Big|.
	\end{multline*}
\end{lemma}
\begin{remark}
Often, when this estimate is used, the sequence $A_\uh$ is defined only on a subset of $\Z^\ell$, and  we  assume that it is extended  to be zero at the elements where it  is not defined.
\end{remark}
\begin{proof}
	The expression on the left-hand side is bounded by
	$$
	\sum_{j=1}^t\sum_{\uh\in I}	\Big|\Big|\sum_{n\in J}\,  w_\uh(n)\cdot  A_\uh(b_{1,\uh}(n)+e_{1,\uh}(n),\ldots, b_{\ell,\uh}(n)+e_{\ell,\uh}(n))	\cdot {\bf 1}_{E_{j,\uh}} (n)\Big|\Big|,
	$$
	where for   $t=f^\ell$ the sets  $E_{1,\uh},\ldots, E_{t,\uh}$ form a partition of $\N$ into sets (possibly empty) on which  all the sequences $e_{1,\uh},\ldots, e_{\ell,\uh}$ are constant (and the constants do not depend on $\uh$). If the maximum of the summands over $j$ occurs for some $j_0\in [t]$, then there exist $\epsilon_1,\ldots, \epsilon_\ell\in F$ such that for all $n\in E_{j_0, \uh}$ we have
	$e_{i,\uh}(n)=\epsilon_i$, $i \in [\ell]$,   $\uh \in I$. Hence, the last sum is bounded by
	$$
	t	\sum_{\uh\in I}\Big|\Big|\sum_{n\in J}\, \tilde{w}_\uh(n)\cdot  A_\uh(b_{1,\uh}(n)+\epsilon_1,\ldots, b_{\ell,\uh}(n)+\epsilon_\ell)	\Big|\Big|,
	$$
	where $\tilde{w}_\uh(n):=w_\uh(n)\cdot {\bf 1}_{E_{j_0,\uh}}(n)$, $n\in J$, $\uh\in I$.
\end{proof}
We will use the previous lemma
 to handle  some error
sequences that occur when we use the Taylor expansion in order  to
perform some approximations
 and when we
replace the sum (or the difference) of the integer parts of sequences with the corresponding integer part of their sum (or the difference), and vice versa. For instance, if $e_1(n),\ldots e_\ell(n)\in (-1,1)$, $n\in [N]$,  we have
\begin{multline*}
	\norm{\E_{n\in [N]}\,  w(n)\cdot  A([a_1(n)+b_1(n)+e_1(n)],\ldots, [a_\ell(n)+b_\ell(n)+e_\ell(n)])}\leq \\
	4^\ell\,
	\norm{\E_{n\in [N]}\,  \tilde{w}(n)\cdot   A([a_1(n)]+[b_1(n)]+\epsilon_1,\ldots, [a_\ell(n)]+[b_\ell(n)]+\epsilon_\ell)},
\end{multline*}
for some $\epsilon_1,\ldots, \epsilon_\ell\in \{-1,0,1,2\}$  and  $\tilde{w}\colon [N]\to \C$ with $\norm{\tilde{w}}_{L^\infty[N]}\leq \norm{w}_{L^\infty[N]}$.  Often the constants $\epsilon_1,\ldots, \epsilon_\ell$ make no difference for our argument and can be ignored.

\section{Seminorm estimates - Sublinear case}
The goal of this section is to establish Theorem~\ref{T:CF'}  in the case where all the iterates have fractional degree smaller than $1$, see Proposition~\ref{P:base} below.
\subsection{An example}\label{Ex1} We explain in some detail how the proof of Theorem~\ref{T:CF'}  works when    $k=1, \ell=2,$ and $a_1(h,t):=p_1(h)t^{0.5}+q_1(h)t^{0.1}$, $a_2(h,t):=p_2(h)t^{0.5}+q_2(h)t^{0.1}$, $h\in\N$, $t\in \R_+$. We assume that $p_1\neq 0$ and $a_1, a_2, a_1-a_2$ are non-zero.

 We also assume that the sequence of weights $(w_{N,h}(n))$  is defined by
$$
w_{N,h}(n):=\Lambda'(n)\cdot \Lambda'(n+h)\cdot c_{N,h}(n), \quad h\in [L_N],\,  n\in[N],\,  N\in \N,
$$
where $(c_{N,h}(n))$ is a $1$-bounded sequence.

 Our aim is to show that there exists $s\in \N$ such that if  $\nnorm{f_1}_s=0$, then
$$
\lim_{N\to\infty}	\E_{h\in [L_N]}\norm{	\E_{n\in [N]}\,  w_{N,h}(n) \cdot \prod_{i=1}^2T^{[p_i(h)n^{0.5}+q_i(h)n^{0.1}]}f_i}_{L^2(\mu)}=0.
$$

{\bf Step 1.} Our first goal is to use the number theory feedback of Section~\ref{SS:sieve},  to reduce
matters to showing mean convergence to zero for some other averages with bounded weights
$w_{N,h}$ (this step corresponds to  Lemma~\ref{L:11} below). We let
$$
p(h):=[\max\{|p_1|(h),|p_2|(h),|q_1|(h),|q_2|(h)\}^{10}]+1, \quad h\in\N.
$$
Note that $p$ is not a polynomial, but this will not  bother us.
After splitting the average over $[N]$ into  subintervals,
 we see (this reduction will be explained in more detail in the proof of Lemma~\ref{L:11}) that it suffices to show    mean convergence to zero for
 $$
 \E_{h\in [L_N]}\norm{\E_{n\in I_{N,h}}\, \E_{n_1\in  J_{n,h}}	\,  w_{N,h}(n_1) \cdot \prod_{i=1}^2 T^{[p_i(h)n_1^{0.5}+q_i(h)n_1^{0.1}]}f_i}_{L^2(\mu)},
 $$
where 
$$
I_{N,h}:= [N^{0.5}p(h)], \quad J_{n,h}:=\Big[\Big(\frac{n-1}{p(h)}\Big)^2,\Big(\frac{n}{p(h)}\Big)^2\Big), \quad n\in I_{N,h}, \, h\in [L_N], \,  N\in\N.
$$

For convenience, we write
$$
J_{n,h}=(k_{n,h},k_{n,h}+l_{n,h}] , \quad
n\in I_{N,h}, \, h\in [L_N], \,  N\in\N,
$$
for some  $k_{n,h}, l_{n,h}\in \N$.

Note that for fixed $n,h\in \N$, when  $n_1$ ranges in $J_{n,h}$
the value of  $p_1(h)n_1^{0.5}$ ranges in an interval of length at most $1$ and the same property holds for the values of
 $p_2(h)n_1^{0.5}$, $q_1(h)n_1^{0.1}$, $q_2(h)n_1^{0.1}$.
Hence,  for $n_1\in J_{n,h}$  we have
$$
p_i(h)n_1^{0.5}+q_i(h)n_1^{0.1}=\frac{p_i(h)}{p(h)}n+q_i(h)\Big(\frac{n}{p(h)}\Big)^{0.2}+e_i(h,n,n_1),\quad i=1,2,
$$
where  $e_1(h,n,n_1) , e_2(h,n,n_1)$ are bounded by $2$  for   all
 $n_1\in J_{n,h}$,   $n\in I_{N,h}$, $h\in [L_N]$, $N\in \N$. Using Lemma~\ref{L:errors},  and since replacing $f_i$ with $T^{\epsilon_{i,N}}f_i$, $i=1,2$,  where $\epsilon_{1,N}, \epsilon_{2,N}$ take finitely many values for $N\in \N$,
does not introduce   changes to our argument,  we can ignore these error terms.  We are thus left with showing  convergence to zero for
 $$
 \E_{h\in [L_N]}\norm{\E_{n\in I_{N,h}}\,   \tilde{w}_{N,h}(n) \cdot \prod_{i=1}^2 T^{\big[\frac{p_i(h)}{p(h)}n+q_i(h) (\frac{n}{p(h)})^{0.2}\big]}f_i}_{L^2(\mu)},
 $$
where for  $n\in [I_{N,h}]$,  $h\in[L_N]$, $N\in \N$, we let
\begin{equation}\label{E:wtilde}
\tilde{w}_{N,h}(n):= \E_{n_1\in  J_{n,h}} \, w_{N,h}(n_1)
= \E_{n_1\in  [l_{n,h}]} \, \Lambda'(n_1+k_{n,h})\cdot \Lambda'(n_1+k_{n,h}+h)\cdot c_{N,h}(n_1+k_{n,h}).
\end{equation}
From the definition of $k_{n,h}$, $l_{n,h}$,  $L_N$, we get that there exists $N_0=N_0(p)\in\N$ such that
$$
k_{n,h}+h\leq l_{n,h}^3,  \quad \text{for all } n\in [N^{0.4},N^{0.5}p(h)], \, h\in [L_N], \, N\geq N_0.
$$
Using Corollary~\ref{C:primes} (with $\ell=1$,  $A=3$, $c=k_{n,h}$, $N=l_{n,h}$) we see that
there exist $D>0$ and $C(h)$, $h\in \N$, such that for the above-mentioned  values of $n,h,N$
we can write
$$
\tilde{w}_{N,h}(n)=C(h)\cdot z_{N,h}(n),
$$
where $(z_{N,h}(n))$ is $1$-bounded  and
$$
\E_{h\in[L_N]}(C(h))^2\leq D
$$
for every $N\in \N$.

We use this estimate, apply the Cauchy-Schwarz inequality, and keep in mind that the part of the intervals $I_{N,h}$ that intersects the interval $[N^{0.4}]$ is negligible for our averages. We deduce that it suffices to  show   convergence to zero  for
$$
\E_{h\in [L_N]}\norm{\E_{n\in I_{N,h}}\,   z_{N,h}(n) \cdot \prod_{i=1}^2 T^{\big[\frac{p_i(h)}{p(h)}n+q_i(h)(\frac{n}{p(h)})^{0.2}\big]}f_i}^2_{L^2(\mu)},
$$
where the sequence  $(z_{N,h}(n))$ is  $1$-bounded.
We write $n=n'p(h)+s$ for some $n'\in  [N^{0.5}]$ and
$s\in  [p(h)]$. For convenience, we also  rename  $n'$ as $n$,
and use Lemma~\ref{L:errors}  to treat  finite valued error sequences that are introduced when we approximate $q_i(h)(n+s/p(h))^{0.2}$ with $q_i(h)n^{0.2}$, $i=1,2$.
 We
get that it suffices to show  convergence to zero  for
$$
\E_{h\in [L_N]}\E_{s\in [p(h)]}\norm{\E_{n\in [N^{0.5}]}\,   z_{N,h,s}(n) \cdot \prod_{i=1}^2 T^{[p_i(h)n+q_i(h)n^{0.2}+e_i(h,s)]}f_i}^2_{L^2(\mu)},
$$
where $(z_{N,h,s}(n))$  is some other  $1$-bounded sequence and $e_i(h,s):=s\frac{p_i(h)}{p(h)}$, $i=1,2$.
After replacing the average $\E_{s\in [p(h)]}$ with $\max_{s\in [p(h)]},$ we are left with dealing with the averages
$$
\E_{h\in [L_N]}\norm{\E_{n\in [N^{0.5}]}\,   z_{N,h}(n) \cdot \prod_{i=1}^2 T^{[p_i(h)n+q_i(h)n^{0.2} +e_{i,N}(h)]}f_i }^2_{L^2(\mu)},
$$
 for some other  $1$-bounded sequence $(z_{N,h}(n))$  and arbitrary sequences of real numbers
 $(e_{1,N}(h)), (e_{2,N}(h))$ (which will be eliminated
later, so their particular form is not important).
\smallskip

 {\bf Step 2.} Our next goal is to reduce matters to showing mean convergence to zero for averages with iterates given by
 polynomials in several variables and real coefficients  (this step corresponds to  Lemma~\ref{L:12} below).
  After using   \eqref{E:VDC1} for the average over $n$, we are left with showing convergence to zero for
\begin{multline*}
	\E_{h,h_1\in [L_N]}\Big|\E_{n\in [N^{0.5}]}\,  c_{N,h,h_1}(n)\cdot  \int
	\prod_{i=1}^2 T^{[p_i(h)(n+h_1)+q_i(h)(n+h_1)^{0.2} +e_{i,N}(h)]}f_i\cdot \\
	\prod_{i=1}^2 T^{[p_i(h)n+q_i(h)n^{0.2} +e_{i,N}(h)]}\overline{f}_i\, d\mu\Big|,
\end{multline*}
where   $(c_{N,h,h_1}(n))$ is a $1$-bounded sequence.
We compose with $T^{-[p_2(h)n+q_2(h)n^{0.2} +e_{2,N}(h)]}$ (and not with $T^{-[p_1(h)n+q_1(h)n^{0.2} +e_{1,N}(h)]}$ because we want the highest  fractional degree iterate to   be applied to the function $f_1$),  use that $(n+h_1)^{0.2}$ can for our purposes be replaced with
$n^{0.2}$, ignore errors that take finitely many values using Lemma~\ref{L:errors},  and use the Cauchy-Schwarz inequality. We are left with showing   convergence to zero for
$$
	\E_{h,h_1\in [L_N]}\norm{\E_{n\in [N^{0.5}]}\, 	c_{N,h,h_1}(n) \cdot
	  T^{[(p_1-p_2)(h)n+(q_1-q_2)(h)n^{0.2} +e_{3,N}(h)]} (T^{[p_1(h)h_1]}f_1\cdot \overline{f}_1)}_{L^2(\mu)},
$$
where $	(c_{N,h,h_1}(n))$ is  some other  $1$-bounded sequence and the sequence  $(e_{3,N}(h))$  takes arbitrary real values.

We consider two cases. Suppose first that  $p_1=p_2$.  Then  by assumption  $q_1-q_2\neq 0$.
 Repeating the argument used in Step 1 we are left with showing convergence to zero for
 $$
 \E_{h,h_1\in [L_N]}\norm{\E_{n\in [N^{0.1}]}\, 	c_{N,h,h_1}(n) \cdot
 	T^{[(q_1-q_2)(h)n +e_{4,N}(h)]} (T^{[p_1(h)h_1]}f_1\cdot \overline{f}_1)}_{L^2(\mu)},
 $$
for some other $1$-bounded sequence of complex numbers  $(c_{N,h,h_1}(n))$ and  $(e_{4,N}(h))$ arbitrary sequence of real numbers.  Using as above
  \eqref{E:VDC1} for the average over $n$,  composing with $T^{-[(q_1-q_2)(h)n +e_{4,N}(h)]}$,
 and then using the Cauchy-Schwarz inequality and Lemma~\ref{L:errors} to treat  errors, we  are  left with  showing mean   convergence to zero for
$$
  \E_{h,h_1,h_2\in [L_N]}\,
	c_{N,h,h_1,h_2}(n)\cdot T^{[(q_1-q_2)(h)h_2+p_1(h)h_1]}f_1\cdot T^{[(q_1-q_2)(h)h_2]}\overline{f}_1 \cdot T^{[p_1(h)h_1]}\overline{f}_1,
$$
for some  $1$-bounded sequence of complex numbers  $(c_{N,h,h_1, h_2}(n))$.

 If $p_1\neq p_2$, we apply  \eqref{E:VDC1} for the average over $n$, then compose with the transformation   $T^{-[(p_1-p_2)(h)n+(q_1-q_2)(h)n^{0.2} +e_{3,N}(h)]}$ and use the Cauchy-Schwarz inequality and Lemma~\ref{L:errors} to treat errors. We are left with showing mean convergence to zero for
$$
\E_{h,h_1,h_2\in [L_N]}\, 	c_{N,h,h_1,h_2}(n) \cdot
	T^{[(p_1-p_2)(h)h_2+p_1(h)h_1]}f_1\cdot T^{[(p_1-p_2)(h)h_2]}\overline{f}_1 	\cdot T^{[p_1(h)h_1]}\overline{f}_1,
$$
for some other $1$-bounded sequence of complex numbers  $(c_{N,h,h_1, h_2}(n))$.
\smallskip

{\bf Step 3.} In  Step 2 we were led to show mean convergence  to zero for averages with iterates  given by non-constant polynomials  with real coefficients in several variables  that have   pairwise non-constant differences. For such averages
one can argue as in \cite{L05} in order to show
 that  there exists $s\in \N$ such that if $\nnorm{f_1}_s=0$, then we have mean convergence to zero. For more details see the proof of  Lemma~\ref{L:13} below.
 This achieves our goal.

\subsection{Reduction to averages with bounded weights and change of variables}
 Our first goal is to prove the following result that allows us to restrict to the case where the weights $w_{N,\uh}$ are $1$-bounded and also allows us to perform
 the substitution $n\mapsto n^{1/d}$.

\begin{lemma}\label{L:11}
	For  $k\in\Z_+,\ell \in \N$, let  $a_1,\ldots, a_\ell$ be a nice collection of fractional polynomials with $k$-parameters  and suppose that $d:=\fdeg(a_1)\in (0,1)$. Then  the following holds:  If $(X,\mu,T)$ is a system,   $f_{N,\uh,1},\ldots, f_{N,\uh,\ell}\in L^\infty(\mu)$, $\uh\in \N^k,N\in\N$,   are $1$-bounded functions, $a>0$, and
	$$
	w_{N,\uh}(n):=(\Delta_\uh\Lambda')(n)\cdot c_{N,\uh}(n) \quad \text{or} \quad
w_{N,\uh}(n):=	c_{N,\uh}(n), \quad  \uh\in [L_N]^k, \, n\in [N^a],\,   N\in\N,
	$$
	where
	$(c_{N,\uh}(n))$ is a  $1$-bounded sequence,
	then      there exist a	$1$-bounded sequence  $(z_{N,\uh}(n))$  and  sequences of real numbers  $(e_{1,N}(\uh)), \ldots, (e_{\ell,N}(\uh))$, such that
	\begin{multline}\label{E:11}
	\E_{\uh\in [L_N]^k}	\norm{\E_{n\in [N^a]}\, w_{N,\uh}(n) \cdot \prod_{i=1}^\ell T^{[a_i(\uh,n)]} f_{N,\uh,i}}_{L^2(\mu)}\ll_{k, a_1,\ldots, a_\ell} \\
	\E_{\uh\in [L_N]^k}	\norm{\E_{n\in [N^{ad}]}\, z_{N,\uh}(n) \cdot \prod_{i=1}^{\ell} T^{[a_i(\uh,n^{1/d})+e_{i,N}(\uh)]} f_{N,\uh,i}}_{L^2(\mu)} +o_N(1),
	\end{multline}
where $o_N(1)$ is a quantity that
converges to  $0$ when $N\to \infty$ and all other parameters remain fixed.
\end{lemma}
\begin{remark}
	It is important   that the function $a_1$ has sublinear growth; our argument would not work if $a_1$ had linear or larger than linear growth.


\end{remark}
\begin{proof}	
 We cover the case where 	
 $w_{N,\uh}(n)=(\Delta_\uh\Lambda')(n)\cdot c_{N,\uh}(n)$, the case where $w_{N,\uh}(n)= c_{N,\uh}(n)$ is similar (in fact easier).
	
By assumption, we have that $a_i(\uh,t):=\sum_{j=0}^r p_{i,j}(\uh)t^{d_j}$, $i=1,\ldots, \ell,$ where $0=d_0< d_1<\ldots<d_r=d< 1$
and $p_{i,j}\in \R[t_1,\ldots, t_k]$ with  $p_{1,r}\neq 0$.   We let
$$
p(\uh):=\big[\max_{i,j}\{|p_{i,j}|(\uh)\}^{\frac{1}{d_1}}\big]+1, \quad \uh\in \N^k.
$$

  For $\uh\in [L_N]^k$, after partitioning $[N^a]$ into sub-intervals,
we deduce that it suffices to get an upper bound for the averages
$$
\E_{\uh\in [L_N]^k}	\norm{\E_{n\in I_{N,\uh}}\, \E^*_{n_1\in J_{n,\uh}}\,
	 w_{N,\uh}(n_1)
	 \cdot \prod_{i=1}^\ell T^{[a_i(\uh,n_1)]} f_{N,\uh,i}}_{L^2(\mu)},
$$
where
$$
I_{N,\uh}:= [N^{ad}p(\uh)], \quad J_{n,\uh}:=\Big[\Big(\frac{n-1}{p(\uh)}\Big)^{\frac{1}{d}},
\Big(\frac{n}{p(\uh)}\Big)^{\frac{1}{d}}\Big), \quad   n\in I_{N,\uh},\, \uh\in [L_N]^k,\, N\in\N,
$$
 and  for $D\colon \N\to \C$  and fixed $N\in\N$,  $n\in I_{N,\uh},\, \uh\in [L_N]^k$, we let
\begin{equation}\label{E:654}
\E^*_{n_1\in J_{n,\uh}}D(n_1) :=\frac{1}{N^a/|I_{N,\uh}|}\sum_{n_1\in J_{n,\uh}}D(n_1).
\end{equation}
Note that an application of the mean value theorem gives  
\begin{equation}\label{E:654'}
|J_{n,\uh}|\leq\frac{1}{d}\frac{(N^{ad}p(h))^{\frac{1}{d}-1} }{p(h)^\frac{1}{d}}=\frac{1}{d} \cdot \frac{N^a}{|I_{N,\uh}|}, \quad n\in I_{N,\uh},\, \uh\in [L_N]^k, \, N\in \N.\footnote{We crucially used here that fractional polynomials do not grow too slowly. The estimate would fail if, for example, for $\ell=1$ we started with $a_1(t): =\log{t}$. }
\end{equation}

For convenience, we write
$$
J_{n,\uh}=(k_{n,\uh},k_{n,\uh}+l_{n,\uh} ], \quad  n\in I_{N,\uh},\,  h\in [L_N]^k,\,  N\in \N,
$$
for some $k_{n,\uh}, l_{n,\uh}\in \N$. Note that for $i=1,\ldots,\ell$, $j=1,\ldots, r$,  and fixed $n,\uh$,  when $n_1$ ranges on $J_{n,\uh}$ the values of $p_{i,j}(\uh)n_1^{d_j}$ belong to an interval of length $1$.
Hence,  for $i=1,\ldots, \ell$ 
  we can write
$$
a_i(\uh,n_1)=a_i(\uh,(n/p(\uh))^{1/d})+\epsilon_i(\uh,n, n_1),
$$
where $\epsilon_i(\uh,n,n_1)$ is bounded by $r$ for all $n_1\in J_{n,\uh}$, $n\in I_{N,\uh}$, $\uh\in [L_N]^k$, $N\in \N$.

The terms $\epsilon_i(\uh,n,n_1)$ can be easily taken care by  using Lemma~\ref{L:errors} and appropriately modifying $(c_{N,\uh}(n))$ to another bounded  sequence of weights.
We deduce that    it suffices to get   an upper bound for the averages
\begin{equation}\label{E:restricted}
\E_{\uh\in [L_N]^k}	\norm{\E_{n\in I_{N,\uh}}\,  {\bf 1}_{I'_{N,\uh}}(n)  \, \tilde{w}_{N,\uh}(n)   \cdot \prod_{i=1}^\ell T^{[a_i(\uh,(n/p(\uh))^{1/d})]+\epsilon_{i,N}} f_{N,\uh,i}}_{L^2(\mu)},
\end{equation}
where  $I'_{N,\uh}:= [N^{\frac{ad}{2}},N^{ad}p(\uh)]$, $N\in \N$ (the indicator introduces a negligible $o_N(1)$ term), $\epsilon_{1,N},\ldots, \epsilon_{\ell,N}$ take finitely many values for $N\in\N$, and  for $n\in [I_{N,\uh}], \uh\in[L_N]^k, N\in\N$, we let
\begin{equation}\label{E:z}
\tilde{w}_{N,\uh}(n):=\E^*_{n_1\in J_{n,\uh}}\,   w_{N,\uh}(n_1).
\end{equation}
 We used that $L_N,\Lambda'(N)\prec N^\varepsilon$ for all $\varepsilon>0$ in order to justify that inserting the indicator ${\bf 1}_{I'_{N,\uh}}$ only introduces  an  $o_N(1)$ term, which is fine for our purposes.

Using that  $(c_{N,\uh}(n))$ is $1$-bounded, $((\Delta_{\uh}\Lambda')(n))$  is non-negative, and \eqref{E:654}, \eqref{E:654'},  we deduce  that
\begin{equation}\label{E:wtilde}
|\tilde{w}_{N,\uh}(n)|\leq d^{-1}\cdot \E_{n_1\in J_{n, \uh}}\,  (\Delta_{\uh}\Lambda')(n_1)=d^{-1}\cdot  \E_{n_1\in [l_{n,\uh}]}\,
(\Delta_{\uh}\Lambda')(n_1+k_{n,\uh}).
\end{equation}
From the definition of $l_{n,\uh}$ and the mean value theorem we have that
$$
l_{n,\uh}\geq \frac{n^{\frac{1}{d}-1}}{d(p(\uh))^{\frac{1}{d}}}, \quad n\in\N.
$$
Since $L_N\prec N^\varepsilon$ for every $\varepsilon>0$ and $k_{n,\uh}\leq n^{1/d}$, it follows  
that if $A>\frac{1}{1-d}$, for example if $A:=\frac{1}{1-d}+1$, then  there exists $N_0=N_0(d,p)\in\N$ such that  for all $N\geq N_0$  and  all  $n\in I'_{N,\uh}$,   $\uh\in [L_N]^k$,
we have
$$
 k_{n,\uh}\leq l_{n,\uh}^A.\footnote{ In the process of deriving this estimate we crucially used that sublinear fractional polynomials are not too close to linear ones. The estimate would fail if, for example, for $\ell=1$ we started with $a_1(t):=t/\log{t}$.}
$$
Hence,  there exists $N_1=N_1(d,k,p)\in\N$ such that  for all $N\geq N_1$ we have for all  $n\in I'_{N,\uh}$ and  $\uh=(h_1,\ldots, h_k)\in [L_N]^k$
that
 $$
 k_{n,\uh}+h_1+\cdots + h_k \leq  l_{n,\uh}^A.
$$
We will combine this with the identity
$$
(\Delta_{\uh}\Lambda')(n)=\prod_{\epsilon\in \{0,1\}^k}\Lambda'(n+\epsilon\cdot \uh),
$$
the estimate  \eqref{E:wtilde}, and Corollary~\ref{C:primes} (with $\ell:=k$, $c:=k_{n,\uh}$, $N:=l_{n,\uh}$, $A:=\frac{1}{1-d}+1$). We deduce  that there exist $C=C(d,k)>0$ and  $C_{d,k}(\uh)>0$, $\uh\in \N^k$,   such that for  all large enough $N$ (depending only on $d,k,p$), for every $n\in I_{N,\uh}$, $\uh\in ([L_N]^k)^*$, we can write
\begin{equation}\label{E:C1}
 \tilde{w}_{N,\uh}(n)=C_{d,k}(\uh)\cdot  z_{N,\uh}(n),
\end{equation}
where $(z_{N,\uh}(n))$ is $1$-bounded and
\begin{equation}\label{E:C2}
\E_{\uh\in [L_N]^k}(C_{d,k}(\uh))^2\leq C
\end{equation}
for every $N\in \N$.

Note that since $L_N\succ (\log{N})^K$ for every $K>0$ and $\Lambda'(n)\leq \log{n}$, for every $n\in\N$   we have that
$\max_{\uh\in [L_N]^k, n\in [N]}(\tilde{w}_{N,\uh}(n))^2\prec L_N$.
 Using this and since by    \eqref{E:card*} we have that
  $\frac{1}{L_N^k}|[L_N]^k\setminus ([L_N]^k)^*|\ll_k \frac{1}{L_N}$, we deduce
  that we can redefine  $C(\uh)$ on the complement of $([L_N]^k)^*$ so
    that for all large enough $N$ (depending on $d,k,p$) equation  \eqref{E:C1} holds for all  $n\in I_{N,\uh}$, $\uh\in [L_N]^k$, and
    \eqref{E:C2} also holds (for some larger constant $C'$ in place of $C$).

We now use  \eqref{E:C1}, \eqref{E:C2}, and the Cauchy-Schwarz inequality,
in order to  bound  the averages in \eqref{E:restricted}. We can also  remove the indicator  $ {\bf 1}_{I'_{N,\uh}}(n)$ since it  has a negligible effect on our averages. We deduce that it suffices to get an upper bound for  the averages
$$
\E_{\uh\in [L_N]^k}	\norm{\E_{n\in I_{N,\uh}}\,   \, z_{N,\uh}(n)   \cdot \prod_{i=1}^\ell T^{[a_i(\uh,(n/p(\uh))^{1/d})]+\epsilon_{i,N}} f_{N,\uh,i}}^2_{L^2(\mu)}.
$$
Note that since the weights and the functions are bounded, it suffices to get an upper bound for the previous expression ignoring the square.
For $\uh\in [L_N]^k$ we can express $n\in I_{N,\uh}$ as $n=n'p(\uh)+s$  for some $n'\in [N^{ad}]$ or $n'=0$ and $s\in [p(\uh)]$.
After renaming  $n'$ as $n$ for convenience, we are led to upper-bounding the averages
\begin{equation}\label{E:789}
 \E_{\uh\in [L_N]^k}\E_{s\in [p(\uh)]}	\norm{\E_{n\in [N^{ad}]}\,   \, z_{N,\uh,s}(n)   \cdot \prod_{i=1}^\ell T^{[a_i(\uh,(n+s/p(\uh))^{1/d})]+\epsilon_{i,N}} f_{N,\uh,i}}_{L^2(\mu)}
 \end{equation}
 for some $1$-bounded sequence $(z_{N,\uh,s}(n))$.
Note that if $u\in (0,1)$ and $q\in\R[t_1,\ldots, t_k]$,  then
an application of the mean value theorem shows that   for every $\varepsilon>0$ we have
$$
\lim_{N\to\infty}\sup_{c\in [0,1],\uh\in[L_N]^k,n\geq N^\varepsilon}|q(\uh)((n+c)^u-n^u)|=0.
$$
 It follows that in \eqref{E:789}
 when computing $a_i(\uh,(n+s/p(\uh))^{1/d})$ we can replace $n+s/p(\uh)$ with $n$ in
  the non-linear monomials; this will lead to some
 error sequences that are $1$-bounded for large enough $N$ and can be handled by appealing to Lemma~\ref{L:errors} (and redefining the sequence $z_{N,\uh}(n)$). With this in mind, it follows that in  \eqref{E:789} we can
 replace $a_i(\uh,(n+s/p(\uh))^{1/d})$ with $a_i(\uh,n^{1/d})+\frac{p_{i,r}(\uh)}{p(\uh)}s$. Hence, it suffices to get an upper  bound for the averages
  $$
 \E_{\uh\in [L_N]^k}\E_{s\in [p(\uh)]}	\norm{\E_{n\in [N^{ad}]}\,   \, z_{N,\uh}(n)   \cdot \prod_{i=1}^\ell T^{[a_i(\uh,n^{1/d})+e_{i,N}(\uh,s)]} f_{N,\uh,i}}_{L^2(\mu)},
 $$
 where $e_{i,N}(\uh,s):=\frac{p_{i,r}(\uh)}{p(\uh)}s+\epsilon_{i,N}$, $i=1,\ldots, \ell$, and $\epsilon_{1,N},\ldots, \epsilon_{\ell,N}$ take finitely many values for $N\in\N$.
 After replacing the average $\E_{s\in [p(\uh)]}$ with $\max_{s\in [p(\uh)]}$ we are led to the asserted upper bound in \eqref{E:11}.
\end{proof}

\subsection{Reduction to averages with polynomial iterates}
For the purposes of the next lemma it will be convenient to slightly enlarge the class of  polynomials with real exponents that we work with  to   include those with  fractional degree equal to  $1$.

\begin{lemma}\label{L:12}
	Let $k\in \Z_+, \ell\in\N$ and $a_1,\ldots, a_\ell\colon \N^k\times \N\to \R$ be a nice collection of polynomials
	with real exponents and $k$-parameters of  fractional  degree at most $1$.
	Then there exist $l,r\in \N$ and non-constant polynomials $P_1,\ldots, P_r\in \R[t_1,\ldots, t_{k+l}]$, with pairwise non-constant differences,  such that the following holds: If $(X,\mu,T)$ is a system and $f_{N,\uh,1},\ldots, f_{N,\uh,\ell}\in L^\infty(\mu)$, $\uh\in \N^k,N\in\N$,  are $1$-bounded functions,
	 then for every $a>0$, sequences of real numbers $(e_{1,N}(\uh)),\ldots, (e_{\ell,N}(\uh))$,  and $1$-bounded sequence of complex numbers $(c_{N,\uh}(n))$,    we have
\begin{multline}\label{E:12}
	\E_{\uh\in [L_N]^k}	\norm{\E_{n\in [N^a]}\, c_{N,\uh}(n) \cdot \prod_{i=1}^\ell T^{[a_i(\uh,n)+e_{i,N}(\uh)]} f_{N,\uh,i}}_{L^2(\mu)}\ll_{k, a_1,\ldots, a_\ell} \\
	\E_{\uh_1\in [L_N]^k,\uh_2\in [L_N]^l} \,\Big| \int \prod_{i=0}^{r} T^{[P_i(\uh_1,\uh_2)]+\epsilon_{i,N}} F_{N,\uh_1, i} \, d\mu\Big| +o_N(1),
	\end{multline}
 where  $P_0:=0$,  $ F_{N,\uh_1, i} \in \{f_{N,\uh_1,1}, \overline{f}_{N,\uh_1,1} \}$ for $i=0,\ldots, r$,  $\uh_1\in [L_N]^k,N\in\N$,  $\epsilon_{0,N},\ldots, \epsilon_{r,N}$ take finitely many values for $N\in\N$,
 and $o_N(1)$ is a quantity that
 converges to  $0$ when $N\to \infty$ and all other parameters remain fixed.
 \end{lemma}

\begin{proof}	
	We first reduce to the case where $e_{i,N}(\uh)=0$ for $i=1, \ldots, \ell$.
	To do this, we  replace $[a_i(\uh,n)+e_{i,N}(\uh)]$ with $[a_i(\uh,n)]+[e_{i,N}(\uh)]$, this introduces
	some error sequences on the exponents that take finitely many values.  To treat the error sequences we use Lemma~\ref{L:errors}, redefine the weight $(c_{N,\uh}(n))$, and introduce
	some   sequences $\epsilon_{1,N}, \ldots, \epsilon_{\ell,N}$ that take finitely many values for $N\in\N$.  Next, we compose with $T^{-[e_{1,N}(\uh)]-\epsilon_{1,N}}$ and
	we are left with upper-bounding the expression
	$$
	\E_{\uh\in [L_N]^k}	\norm{\E_{n\in [N^a]}\, c_{N,\uh}(n) \cdot \prod_{i=1}^\ell T^{[a_i(\uh,n)]} (T^{[e_{i,N}(\uh)]-[e_{1,N}(\uh)]+\epsilon_{i,N}-\epsilon_{1,N}}f_{N,\uh,i})}_{L^2(\mu)}.
	$$
	If we rename  for $i=2,\ldots, \ell$ the functions
	 $T^{[e_{i,N}(\uh)]-[e_{1,N}(\uh)]+\epsilon_{i,N}-\epsilon_{1,N}}f_{N,\uh,i}$
as $f_{N,\uh,i}$, we are reduced to bounding \eqref{E:12} when $e_{i,N}(\uh)=0$ for $i=1, \ldots, \ell$.	
	
	We will  prove the statement by induction on $\ell\in \N$.
		For $\ell=1$ the argument is similar to the one used in the inductive step, so  we only summarise it briefly (for more details see Steps~1-3 below).
		We first use Lemma~\ref{L:11} and we are led to upper-bounding the averages
		$$
	\E_{\uh\in [L_N]^k}	\norm{\E_{n\in [N^a]}\, c_{N,\uh}(n) \cdot  T^{[p_1(\uh)n+q_1(\uh,n)]} f_{N,\uh,1}}_{L^2(\mu)},
	$$
	where $p_1\neq 0$ and  $q_1$ is a polynomial with real exponents and $\fdeg(q_1)<1$. We  then 		
		apply  \eqref{E:VDC1} for the average over $n$, compose with $T^{-[p_1(\uh)n+q_1(\uh,n)]}$, use that
		$q_1(\uh,n+h_{k+1})-q_1(\uh,n)$ is negligible  for the range of parameters we are interested in, and use Lemma~\ref{L:errors} to treat the finite valued error sequences that arise. We get an upper bound by the averages
		$$
	\E_{\uh\in [L_N]^k, h_{k+1}\in [L_N]}	\Big|  \int T^{[p_1(\uh)h_{k+1}]+\epsilon_N} f_{N,\uh,1} \cdot  \overline{f}_{N,\uh,1} \, d\mu\Big|,
	$$
	where $\epsilon_N$ takes finitely many values for $N\in\N$. This proves \eqref{E:12} (with $\ell=r=1$).

	Suppose that $\ell\geq 2$ and  the statement holds for all nice collections of $\ell-1$ polynomials with real exponents and finitely many parameters.
	
We have that $a_i(\uh,t):=\sum_{j=1}^r p_{i,j}(\uh)t^{d_j}$, $i=1,\ldots, \ell$ where $0\leq d_1<\cdots<d_r=d\leq 1$
and $p_{i,j}\in \R[t_1,\ldots, t_k]$.  Furthermore, we can assume that the polynomial
$p_{1,r}$ is non-zero, and hence the fractional degree of $a_1$ is $d$.
\smallskip

{\bf Step 1} (Linearising the highest-order term).
If  the fractional degree of $a_1$ is $1$, then we proceed to Step 2.
If not, then Lemma~\ref{L:11}  (for $w_{N,\uh}:=c_{N,\uh}$) applies and we get an estimate of the form \eqref{E:11}.
Hence, in order to get an estimate of the form \eqref{E:12}, it suffices to get a similar  estimate  for the averages
$$
\E_{\uh\in [L_N]^k}	\norm{\E_{n\in [N^{ad}]}\, c_{N,\uh}(n) \cdot \prod_{i=1}^{\ell} T^{[\tilde{a}_i(\uh,n)]+e_{i,N}(\uh)} f_{N,\uh,i}}_{L^2(\mu)},
$$
where $(c_{N,\uh}(n))$ is another $1$-bounded sequence, $(e_{1,N}(\uh)), \ldots, (e_{\ell,N}(\uh))$ are sequences of real numbers, and
	\begin{equation}\label{E:subl1}
\tilde{a}_i(\uh,t):=p_{i,r}(\uh)\,t+ q_i(\uh,t), \quad \text{where} \quad q_i(\uh,t):=\sum_{j=1}^{r-1} p_{i,j}(\uh) \, t^{\frac{d_j}{d}}, \quad i=1,\ldots, \ell.
\end{equation}
After composing with $T^{-e_{1,N}(\uh)}$ and redefining the functions $f_{N,\uh,i}$, $i=2,\ldots, \ell$, we are reduced to the case where $e_{i,N}(\uh)=0$ for $i=1,\ldots, \ell$. So we only treat this case henceforth.
We also remark that since the collection $a_1,\ldots, a_\ell$ is nice, and $\tilde{a}_i(\uh,t)=a_i(\uh,t^{1/d})$, $i=1,\ldots, \ell$,  the collection $\tilde{a}_1,\ldots, \tilde{a}_\ell$ is also nice.
\smallskip

{\bf Step 2} (Reduction of $\ell$ via vdC).
Applying \eqref{E:VDC1} for  the average over $n$, we get that  it suffices to obtain an upper bound for the following averages
$$
\E_{(\uh, h_{k+1})\in [L_N]^{k+1}}		\E_{n\in [N^{ad}]}\Big| \int \prod_{i=1}^\ell T^{[\tilde{a}_i(\uh,n+h_{k+1})]} f_{N,\uh,i}\, \prod_{i=1}^\ell T^{[\tilde{a}_i(\uh,n)]} \overline{f}_{N,\uh,i}\,
\,
d\mu\Big|.
$$

We compose with $T^{-[\tilde{a}_1(\uh,n)]}$ and for $i=1,\ldots, \ell$ we  replace the differences $[\tilde{a}_i(\uh,n+h_{k+1})]-[\tilde{a}_1(\uh,n)]$,
$[\tilde{a}_i(\uh,n)]-[\tilde{a}_1(\uh,n)]$ with $[\tilde{a}_i(\uh,n+h_{k+1})-\tilde{a}_1(\uh,n)]$, $[\tilde{a}_i(\uh,n)-\tilde{a}_1(\uh,n)]$, respectively.
To do so, we have to introduce some error sequences that take values on a finite subset of $\N$.  We use Lemma~\ref{L:errors} to treat the errors  that arise and  we are left with upper-bounding  averages of the form
\begin{multline*}
\E_{(\uh, h_{k+1})\in [L_N]^{k+1}}		\E_{n\in [N^{ad}]}\Big| \int \prod_{i=1}^\ell T^{[\tilde{a}_i(\uh,n+h_{k+1})-\tilde{a}_1(\uh,n)]+\epsilon_{i,N}} f_{N,\uh,i}\cdot
\\ \prod_{i=1}^\ell T^{[\tilde{a}_i(\uh,n)-\tilde{a}_1(\uh,n)]+\epsilon'_{i,N} } \overline{f}_{N,\uh,i}\,
\,
d\mu\Big|,
\end{multline*}
where $\epsilon_{i,N},\epsilon'_{i,N}$, $i=1,\ldots, \ell$, take finitely many values for $N\in\N$.
Note  that the fractional degree of $q_1, \ldots, q_\ell$ is strictly smaller than $1$. It follows from this and the mean value theorem that
\begin{equation}\label{E:subl2}
\lim_{N\to\infty}\max_{(\uh,h_k)\in [L_N]^{k+1}, i\in \{1,\ldots, \ell\}}|q_i(\uh,t+h_{k+1})-q_i(\uh,t)|=0.
\end{equation}
 Using \eqref{E:subl1} and \eqref{E:subl2}, and  then  Lemma~\ref{L:errors}, we get that it suffices to get an upper bound for the averages
\begin{multline*}	
\E_{(\uh,h_{k+1})\in [L_N]^{k+1}}	\E_{n\in [N^{ad}]}\,  c_{N,\uh,h_{k+1}}(n) \cdot \int
(T^{[p_{1,r}(\uh)h_{k+1}]+\epsilon_{1,N} }f_{N,\uh,1}\cdot \overline{f}_{N,\uh,1})\cdot \\ \prod_{i=2}^\ell T^{[b_i(\uh,n)]+\epsilon_{i,N} }\tilde{f}_{N,\uh,h_{k+1},i}\, d\mu,
\end{multline*}
where $\epsilon_{1,N},\ldots, \epsilon_{\ell,N}$ take finitely many values for $N\in\N$,
$$
b_{i}(\uh,t):=(p_{i,r}-p_{1,r})(\uh)\, t+(q_{i}-q_1)(\uh,t), \quad i=2,\ldots, \ell,
$$
  $\tilde{f}_{N,\uh,h_{k+1},i}\in L^\infty(\mu)$, $i=2,\ldots, \ell$, are  $1$-bounded functions,   and
 $(c_{N,\uh,h_{k+1}}(n))$ is a $1$-bounded sequence.
 Without loss of generality we can assume that  $b_\ell$ has maximal fractional degree within the collection $b_2,\ldots, b_\ell$ (note that some of the polynomials $p_{i,r}-p_{1,r}$ may vanish).
 We compose with $T^{-[b_\ell(\uh,n)]}$ and  apply Lemma~\ref{L:errors} to treat  finite-valued error sequences that we get when we replace
differences of integer parts with the integer part of the corresponding differences. After using the Cauchy-Schwarz inequality we deduce that it suffices to get an upper bound for the following averages
\begin{equation}\label{E:ready}
\E_{h_{k+1}\in [L_N]} \Big(\E_{\uh\in [L_N]^k} \norm{	\E_{n\in [N^{ad}]} \, c_{N,\uh,h_{k+1}}(n) \,  \prod_{i=1}^{\ell-1} T^{[\tilde{b}_i(\uh,n)]+\epsilon'_{i,N}}\tilde{f}_{N,\uh,h_{k+1},i}}_{L^2(\mu)}\Big),
\end{equation}
where $ \epsilon'_{1,N},\ldots, \epsilon'_{\ell-1,N}$ take finitely many values for $N\in\N$,
$$
\tilde{b}_{i}(\uh,t):=(p_{i,r}-p_{\ell,r})(\uh)\, t+(q_{i}-q_\ell)(\uh,t), \quad i=1,\ldots, \ell-1,
$$
and
\begin{equation}\label{E:replace}
\tilde{f}_{N,\uh,h_{k+1},1}:=
T^{[p_{1,r}(\uh)h_{k+1}]+\epsilon_{1,N}}f_{N,\uh, 1}\cdot \overline{f}_{N,\uh, 1},
\end{equation}
where $ \epsilon_{1,N}$ takes finitely many values for $N\in\N$.

Note that our assumptions imply that  $\tilde{b}_{1},\ldots,  \tilde{b}_{\ell-1}$, thought of as a collection  of polynomials  with real exponents and $(k+1)$-parameters, is nice.
\smallskip

{\bf Step 3} (Applying the induction hypothesis).  Using the induction hypothesis  for the expression in \eqref{E:ready} that is inside the parentheses, and the fact that $\tilde{b}_1,\ldots, \tilde{b}_{\ell-1}$ do not depend on the parameter $h_{k+1}$,
we get that
there exist $l,r\in \N$ and  non-constant  polynomials $P_1,\ldots, P_r\in \R[t_1,\ldots, t_{k+l}]$ with pairwise non-constant differences,    such that   the averages in  
\eqref{E:ready} are bounded
 by an $o_N(1)$ term plus a constant $C_{k,a_1,\ldots, a_\ell}$
(note that $\tilde{b}_{1},\ldots,  \tilde{b}_{\ell-1}$ are determined by $a_1,\ldots, a_\ell$) times the expression
$$
\E_{(\uh_1,h_{k+1})\in [L_N]^{k+1},\uh_2\in [L_N]^l}\Big| \int \prod_{i=0}^{r} T^{[P_i(\uh_1,\uh_2)]+\epsilon'_{i,N}}
	F_{N,\uh_1, h_{k+1},i}\, d\mu\Big|,
$$
where $P_0:=0$, $ F_{N,\uh,h_{k+1},i} \in \{ \tilde{f}_{N,\uh,h_{k+1},1}, \overline{\tilde{f}}_{N,\uh,h_{k+1},1} \}$ for $i=0,\ldots, r$,  $\uh_1\in [L_N]^k$, $h_{k+1}\in [L_N]$, $N\in \N$,
and $\epsilon'_{0,N},\ldots, \epsilon'_{r,N}$ take finitely many values for $N\in\N$.

Using \eqref{E:replace} and Lemma~\ref{L:errors}, we can bound this expression by a constant $C_{r}$ times
the  following average
\begin{multline*}
	\E_{(\uh_1,h_{k+1})\in [L_N]^{k+1},\uh_2\in [L_N]^l} \Big|\int f_{N,\uh_1, 1} \cdot T^{[p_{1,r}(\uh_1)h_{k+1}]+\epsilon'_{0,N}}\overline{f}_{N,\uh_1, 1}\cdot\\
			 \prod_{i=1}^{r} \big(
			T^{[P_i(\uh_1,\uh_2)+p_{1,r}(\uh_1)h_{k+1}]+\epsilon'_{i,N}} G_{N,\uh_1,h_{k+1},i}\cdot T^{[P_i(\uh_1,\uh_2)]+\epsilon'_{r+i,N}}  G_{N,\uh_1,h_{k+1},r+i}\big)\, d\mu\Big|,
\end{multline*}
where for $i=1,\ldots, 2r$ we have  $G_{N,\uh_1,h_{k+1},i}  \in \{ f_{N,\uh_1,1}, \overline{f}_{N,\uh_1,1} \}$,
$\uh_1 \in [L_N]^k$, $h_{k+1}\in [L_N]$, $N\in\N$, and  $\epsilon'_{i,N}$, $i=0,\ldots, 2r$, take finitely many values for $N\in\N$.
Since the polynomial $p_{1,r}$ is non-zero and the polynomials $P_1,\ldots, P_r$ with $k+l$ variables are non-constant and have non-constant pairwise differences, the same holds for the $2r+1$ polynomials with   $k+l+1$ variables
$p_{1,r}(\uh_1)h_{k+1}$, $P_i(\uh_1,\uh_2)+p_{1,r}(\uh_1)h_{k+1}$, $P_i(\uh_1,\uh_2)$, $i=1,\ldots, r$.  This completes the proof.
\end{proof}

\subsection{Averages with polynomial iterates}
Lemma~\ref{L:11} and Lemma~\ref{L:12} show that in the case of iterates with  sublinear growth,  to get good seminorm estimates for the averages in Theorem~\ref{T:CF'}, it suffices to study averages with iterates given by polynomials in $\R[t_1,\ldots, t_k]$ for some $k\in\N$. This is the context of the next  result.
\begin{lemma}\label{L:13}
Let $k,r\in\N$ and $P_1,\ldots, P_r\in\R[t_1,\ldots, t_k]$ be non-constant polynomials with pairwise non-constant differences. Then there exists $s\in \N$ such that the following holds: If $(X,\mu,T)$ is an ergodic  system and $f_1,\ldots, f_r\in L^\infty(\mu)$ are such that $\nnorm{f_i}_s=0$ for some $i\in \{1,\ldots, r\}$, then for every $1$-bounded sequence $(c_N(\uh))$,  we have
$$
\lim_{N\to\infty} \E_{\uh\in [N]^k}\, c_{N}(\uh)\cdot \prod_{i=1}^{r} T^{[P_i(\uh)]} f_i=0
$$
in $L^2(\mu)$.
\end{lemma}
\begin{proof}
	The argument is similar to the one used to prove \cite[Theorem~1]{L05} where the case of  polynomials with  integer coefficients  and $c_{N}(\uh):=1$ is covered, so we only sketch the points in the argument where one has to  deviate slightly because of minor technical complications. The proof proceeds by  induction on
	a certain vector, called the weight, that is associated to each polynomial family $P_1,\ldots, P_r$  in $\R[t_1,\ldots, t_k]$.

		The inductive step is carried out by using a variant of Lemma~\ref{L:VDC} in the form used in \eqref{E:VDC2} that  concerns  averages over $[N]^k$ (see \cite[Lemma~4]{L05} for the precise statement). The argument applies verbatim in our case, the only change is that we need at various instances to replace the differences of the integer part of polynomials with the integer part of their differences; we do this with the help of Lemma~\ref{L:errors} and  the use of the constants $(c_N(\uh))$ facilitates this task.
	
	 The base case of the induction is the case
	where all the polynomials are linear with respect to all   variables involved. This case is covered using another induction, this time  on the number $r$ of linear functions. The inductive step is proved using \cite[Lemma~4]{L05}. The only difference in our case, versus the argument used in \cite[Proposition~5]{L05}, appears in the proof of the estimate
	\begin{equation}\label{E:L}
	\limsup_{N\to\infty}\E_{\uh\in [N]^k}\nnorm{g\cdot T^{[L(\uh)]}f}^{2^s}_{s}\leq
	C_L\, \nnorm{f}^{2^{s+1}}_{s+1},
	\end{equation}
for some $C_{L}>0$, where $f,g\in L^\infty(\mu)$ and $L(\uh)=\sum_{j=1}^k\alpha_j h_j$, for some $k\in\N$ and $\alpha_1,\ldots, \alpha_k\in \R$. To obtain this bound,  we first use Lemma~\ref{L:errors}  to show that it suffices to  replace
$[\sum_{j=1}^k\alpha_j h_j]$ with $\sum_{j=1}^k[\alpha_j h_j]$, and  we  remark that the set
$$
\{([\alpha_1 h_1],\ldots, [\alpha_k h_k])\colon (h_1,\ldots, h_k)\in \N^k\}
$$ has bounded multiplicity
and positive density (as a subset of $\N^k$). It follows that there exists $C_L>0$ such that
$$
\limsup_{N\to\infty}\E_{\uh\in [N]^k}\nnorm{g\cdot T^{\sum_{j=1}^k[\alpha_j h_j]}f}^{2^s}_{s}\leq
C_L\,  \limsup_{N\to\infty}\E_{\uh\in [N]^k}\nnorm{g\cdot T^{\sum_{j=1}^k h_j}f}^{2^s}_{s}.
$$	
By \cite[Lemma~8]{L05}, the last expression is bounded by a constant multiple of $\nnorm{f}^{2^{s+1}}_{s+1}$. Combining the above we  get that \eqref{E:L} holds.
Finally, the base case of the induction (of the linear case) is when  $r=1$ and $P_1=L$ is linear.  To cover this case,  we again use  \cite[Lemma~4]{L05}
 and  reduce matters to the task of obtaining  an upper bound for the expression
$$
 \limsup_{N\to\infty}\E_{\uh\in [N]^k}\Big|\int \overline{f}\cdot T^{L(\uh)}f\, d\mu\Big|.
 $$	
 By the $s=1$ case of \eqref{E:L} (recall that $\nnorm{f}_1=|\int f\, d\mu|$) we get an upper bound by $C_L\nnorm{f}_2^2$ for some $C_L>0$. This completes the proof.
	\end{proof}


\subsection{Proof of Theorem~\ref{T:CF'} in the sublinear case}
We are now ready to combine the ingredients of the previous subsections  to complete the goal of this section, which is to prove the following result:
\begin{proposition}\label{P:base}
Theorem~\ref{T:CF'} holds in the case where all   $a_1,\ldots, a_\ell$ have fractional degree smaller than one.
\end{proposition}
\begin{proof}	
Combining Lemma~\ref{L:11} and Lemma~\ref{L:12} (for $f_{N,\uh,1}:=f_1$, $N\in\N, \uh\in\N^k$) we get that
there exist $k,r\in \N$ and non-constant polynomials $P_1,\ldots, P_r\in\R[t_1,\ldots, t_{k}]$, with pairwise non-constant differences, such that  the averages \eqref{E:mainav} are bounded by an $o_N(1)$ term plus a constant multiple of
$$
\E_{\uh\in [L_N]^{k}}\Big|\int \prod_{i=0}^{r} T^{[P_i(\uh)]+ \epsilon_{i,N}} F_{i,\uh}\, d\mu\Big|,
$$
where $P_0:=0$,  $F_{0,\uh},\ldots, F_{r,\uh}\in \{f_1,\overline{f}_1\}$, $\uh\in\N^k$,  and the sequences $\epsilon_{0,N},\ldots, \epsilon_{r,N}$ take values on a finite subset $S$ of $\Z$ for $N\in\N$.
Since the limsup as $N\to \infty$ of the previous average is bounded by
$$
\sum_{\epsilon_0,\ldots, \epsilon_r\in S,\,  F_0,\ldots, F_r\in \{f_1,\overline{f}_1\}}\limsup_{N\to\infty}\Big(\E_{\uh\in [L_N]^{k}}\Big|\int \prod_{i=0}^{r} T^{[P_i(\uh)]+ \epsilon_{i}} F_{ i}\, d\mu\Big|\Big),
$$
 it suffices to show that for all fixed  $\epsilon_0,\ldots, \epsilon_r\in \Z$ and $F_{0},\ldots, F_{r}\in \{f_1,\overline{f}_1\}$, we have
$$
\lim_{N\to\infty}\E_{\uh\in [L_N]^{k}}\Big|\int \prod_{i=0}^{r} T^{[P_i(\uh)]+ \epsilon_{i}} F_{ i}\, d\mu\Big|=0.
$$
The last average  is equal to
$$
\E_{\uh\in [L_N]^{k}}\,  c_N(\uh)\cdot \int \prod_{i=0}^{r} T^{[P_i(\uh)+\epsilon_i]} F_{ i}\, d\mu
$$
 for some  $1$-bounded sequence $(c_N(\uh))$.
The result now  follows from Lemma~\ref{L:13}.
\end{proof}

\section{Seminorm estimates - induction step}
The goal of this section is to finish the proof of  Theorem~\ref{T:CF'} using a PET-induction argument.
The  basis of the induction was covered in the previous section and the induction step will be carried out in this section.
\subsection{An example} \label{Ex2}
To better illustrate our method,  we first  explain the details in a  simple case. We take   $\ell=2$ and $a_1(t):=t^{1.5}, a_2(t)=t^{1.5}+ t^{1.1}$, $t\in \R_+$. Then $\{a_1,a_2\}$ is a nice family and our aim is to show that if $\nnorm{f_1}_s=0$ for some $s\in \N$,   then
$$
\lim_{N\to\infty}		\E_{n\in [N]}\,  w_N(n) \cdot T^{[n^{1.5}]}f_1\cdot T^{[n^{1.5}+n^{1.1}]}f_2=0,
$$
where $w_N(n)=\Lambda'(n)\cdot c_N(n)$ for some $1$-bounded sequence
$(c_N(n))$.

We start by using 
\eqref{E:VDC1}, compose with $T^{-[n^{1.5}+n^{1.1}]}$, use Lemma~\ref{L:errors}    to dispose the error sequence  that arises when  we replace the difference of integer parts with the integer part of the difference, and use the Cauchy-Schwarz inequality.  We deduce  that  it suffices to prove convergence to zero of   the averages
\begin{multline*}
\E_{h_1\in [L_N]}\Big|\Big|\E_{n\in [N]} \,   w_{N,h_1}(n) \cdot
T^{[(n+h_1)^{1.5}-n^{1.5}-n^{1.1}]}f_1\cdot \\ T^{[(n+h_1)^{1.5}+(n+h_1)^{1.1}-n^{1.5}-n^{1.1}]}f_2 \cdot
T^{[-n^{1.1}]}\overline{f}_1\Big|\Big|_{L^2(\mu)},
\end{multline*}
where
$w_{N,h_1}(n):=
(\Delta_{h_1}\Lambda')(n)\cdot  c_{N,h_1}(n)$ for some  $1$-bounded sequence  $(c_{N,h_1}(n))$.
Using the mean value theorem and Lemma~\ref{L:errors}, we get that for  the range of $h_1,n$ we are working with, we can replace
$(n+h_1)^{1.5}-n^{1.5}$ with $1.5\, h_1n^{0.5}$ and $(n+h_1)^{1.1}-n^{1.1}$ with $1.1\, h_1n^{0.1}$, which for notational simplicity we replace with $h_1n^{0.5}$ and $h_1n^{0.1}$ respectively. We thus arrive to the problem of proving convergence to
zero of the averages
$$
\E_{h_1\in [L_N]}\norm{\E_{n\in [N]} \, w_{N,h_1}(n)\cdot
	T^{[-n^{1.1}+h_1n^{0.5}]}f_1\cdot T^{[h_1n^{0.5}+h_1 n^{0.1}]}f_2 \cdot
	T^{[-n^{1.1}]}\overline{f}_1}_{L^2(\mu)}.
$$
Performing one more time the previous operation (we compose with $T^{-[h_1n^{0.5}+h_1 n^{0.1}]}$ after  applying  \eqref{E:VDC1}) we arrive in a similar fashion
at the following averages
\begin{multline*}
\E_{h_1,h_2\in [L_N]}\Big|\Big|\E_{n\in [N]} \,w_{N,h_1,h_2}(n) \cdot
 T^{[-n^{1.1}-h_2 n^{0.1}]}f_1  \cdot
 T^{[-n^{1.1}-h_1n^{0.5}-(h_1+h_2) n^{0.1}]}\overline{f}_1\cdot \\
	T^{[-n^{1.1}-h_1 n^{0.1}]}\overline{f}_1 \cdot
	T^{[-n^{1.1}-h_1n^{0.5}-h_1 n^{0.1}]}f_1\Big|\Big|_{L^2(\mu)},
\end{multline*}
where   $w_{N,h_1,h_2}(n):= (\Delta_{h_1,h_2}\Lambda')(n)\cdot c_{N,h_1,h_2}(n)$ for some $1$-bounded sequence $(c_{N,h_1,h_2}(n))$.   After one more iteration of the previous operation  (this time we  compose with the transformation $T^{[n^{1.1}+h_1n^{0.5}+h_1 n^{0.1}]}$ after applying \eqref{E:VDC1}) we arrive at the averages
\begin{multline*}
\E_{h_1,h_2,h_3\in [L_N]}
\Big|\Big| \E_{n\in [N]}\,   w_{N,h_1,h_2,h_3}(n) \cdot  \
 T^{[(h_1-h_2-h_3) n^{0.1}+h_1n^{0.5}]}f_1  \cdot
T^{[-(h_2+h_3) n^{0.1}]}\overline{f}_1\cdot \\
T^{[-h_3 n^{0.1}+h_1n^{0.5}]}\overline{f}_1 \cdot
T^{[-h_3 n^{0.1}]}f_1\cdot
T^{[(h_1-h_2) n^{0.1}+h_1n^{0.5}]}\bar{f}_1  \cdot
T^{[-h_2 n^{0.1}]}f_1\cdot
T^{[h_1n^{0.5}]}f_1
 \Big|\Big|_{L^2(\mu)},
\end{multline*}
where
  $w_{N,h_1,h_2,h_3}(n):= (\Delta_{h_1,h_2,h_3}\Lambda')(n)\cdot c_{N,h_1,h_2,h_3}(n)$
  for some $1$-bounded   sequence $(c_{N,h_1,h_2,h_3}(n))$.
 We have now reduced to the case of fractional polynomials with $3$-parameters and fractional degree smaller than $1$. This case was dealt  in the previous section, where we showed in Proposition~\ref{P:base} that there exists $s\in \N$ such that if $\nnorm{f_1}_s=0$, then the last averages converge to zero as $N\to\infty$.

\subsection{The van der Corput operation and reduction of type}\label{SS:VDC}
 In this subsection we define the type of a family of polynomials with real exponents
 and  finitely many parameters and the van der Corput  operation that reduces the type.
\begin{definition}
We say that two   polynomials  $a, b$ with real exponents and finitely many parameters are {\em equivalent}, and
write   $a\cong b$, if the (integral) degree of  $a- b$ is strictly smaller than the degree of  $a$ and  $b$.\footnote{We do not choose to identify functions with the same fractional degree because if we did so, then  the vdC operation that will be described shortly
	would not necessarily lead to families with smaller type (see the example given after the relevant definition).}

We define the {\em type} of  a family $a_1,\ldots, a_\ell$ of  polynomials with real exponents and  finitely many parameters to be  the vector that consists of the maximal degree  $d$ of the family (in the first coordinate) and the number of non-equivalent classes of degree $d$, $d-1$,$\ldots$, $0$ in the other coordinates (we ignore polynomials that are identically $0$).

We order the set of all possible types lexicographically; meaning, $(d, k_d,\ldots, k_0)>(d', k_d',\ldots, k_0')$ if and only if in the first instance where the two vectors disagree the coordinate of the first vector is larger than the coordinate of the second vector.
\end{definition}
We caution the reader that $t^{2.5}\not\cong t^{2.5}+t^{2.1}$ (but  $t^{2.5}\cong t^{2.5}+t^{1.1}$).
Also   if  $a_1(h,t)=ht^{2.5}+h^2t^{2.1}$,  $a_2(h,t)=ht^{2.5}$,  $a_3(h,t)=ht^{2.5}+h^2t^{2.1}+ht^{1.5}$, $a_4(h,t)=t^{0.5}$,  then $a_1\not\cong a_2$, $a_2 \not\cong a_3$, $a_1\cong a_3$
and the family $a_1,a_2,a_3, a_4$ has type $(2,2,0,1)$.

Recall that
$L_N=[e^{\sqrt{\log{N}}}]$, $N\in\N$.  We  introduce a class of
sequences that often occur as errors that can be  eliminated using Lemma~\ref{L:errors}.
\begin{definition}
We say that $e\colon \N^k\times \R_+\to \R$ is {\em negligible} if
	$$
\lim_{N\to\infty} \max_{\uh\in [L_N]^k, t\in [\sqrt{N},N]}|e(\uh,t)|=0.
$$
\end{definition}
If $a(t)$ is a fractional polynomial, then $a(t+c)$ is also a fractional polynomial modulo negligible terms. This is the context of the next lemma, which is proved in a more general form that is better suited for our purposes.
\begin{lemma}\label{L:shifts}
	Let $a(\uh,t)$ be a polynomial with real exponents and $k$-parameters and degree $d$. Then
	modulo negligible terms,	$a(\uh,t+h_{k+1})$ is  a polynomial with real exponents and $(k+1)$-parameters. In fact, we have
	\begin{equation}\label{E:tildea}
	a(\uh,t+h_{k+1})=	\tilde{a}(\uh,h_{k+1},t)+e(\uh,h_{k+1},t),
	\end{equation}
	where (below $a^{(j)}$ denotes the $j$-th derivative of $a$ with  respect to the variable $t$)
	\begin{equation}\label{E:approx}
	\tilde{a}(\uh,h_{k+1},t):=\sum_{j=0}^d \frac{h_{k+1}^j}{j!} a^{(j)}(\uh,t)
	\end{equation}
	and $e\colon \N^{k+1}\times \R \to \R$ is negligible.
\end{lemma}
\begin{proof}
	Using the Taylor expansion of $a(\uh,t)$ we get that \eqref{E:tildea} holds with
	$$
	e(\uh,h_{k+1},t):=  \frac{h_{k+1}^{d+1}}{(d+1)!}\,  a^{(d+1)}(\uh,\xi_{\uh,h_{k+1},t})
	$$
	for some  $\xi_{\uh,h_{k+1},t}\in [t,t+h_{k+1}]$. Since the fractional degree of $a$ is $d+c$ for some $c\in (0,1)$, we have
	$$
\max_{(\uh,h_{k+1})\in [L_N]^{k+1}, t\in [\sqrt{N},N]}|		e(\uh,h_{k+1},t)|\prec \frac{L_N^A}{N^{\frac{1-c}{2}}},
	$$
	for some $A>0$ that depends on $d$ and the maximum degree of the coefficient polynomials of $a(\uh,t)$. Since $L_N\prec N^\varepsilon$ for every $\varepsilon>0$,  it follows that
	$$
	\lim_{N\to\infty}\max_{(\uh,h_{k+1})\in [L_N]^{k+1}, t\in [\sqrt{N},N]}|		e(\uh,h_{k+1},t)|=0,
	$$
	 completing the proof.
	\end{proof}
For example, if $a(h,t)=ht^a$ for some $a\in (2,3)$, then modulo negligible terms (in the sense defined above)
we have that $\tilde{a}(h,t+h_{1})$ is equal to $ht^a+ah_1ht^{a-1}+\frac{a(a-1)}{2}h_1^2ht^{a-2}$.


Next we define an operation that we later show preserves nice families of polynomials and reduces their type.
 \begin{definition}
Let  $\mathcal{A}=\{a_1,\ldots, a_\ell\}$ be a family of polynomials with real exponents and $k$-parameters and $a\in \mathcal{A}$. We define a new family of polynomials with real exponents and $(k+1)$-parameters $\text{vdC}(\mathcal{A},a)$ as follows: We start with the family
$$
\{ \tilde{a}_i(\uh, h_{k+1},t)-a(\uh,t), \, a_i(\uh,t)-a(\uh,t),\,  i=1,\ldots, \ell
\},
$$
where  for $i=1,\ldots, \ell$  the polynomial with real exponents and $(k+1)$-parameters $\tilde{a}_i$ is as in \eqref{E:approx}  (so it is equal to 	$a_i(\uh,t+h_{k+1})$ modulo negligible terms), and we  remove all functions that are constant in the variable $t$.
\end{definition}
Suppose for example that we start with the nice family
$$
\mathcal{A}=\{t^{1.5}, t^{1.5}+t^{1.1}, t^{1.5}+t^{1.2}\}.
$$
The type of this family is $(1, 3,0)$ and the family  $\text{vdC}(\mathcal{A},t^{1.5}+t^{1.2})$ is
$$
\{-t^{1.2}+1.5ht^{0.5}, -t^{1.2}+t^{1.1}+1.5ht^{0.5}, 1.5ht^{0.5}+1.2ht^{0.2}, -t^{1.2}, -t^{1.2}+t^{1.1}\},
$$
(note that the first and fourth functions can be identified and the same holds for the second and the fifth)
which is also nice and has smaller type, namely  $(1,2, 1 )$. We remark that if we had chosen to  identify  functions
that have   the same fractional degree, then  the original family would have type
$(1,1,0)$ and the family $\text{vdC}(\mathcal{A},t^{1.5}+t^{1.2})$  would have  larger type, namely $(1,2,1)$.
\begin{lemma}\label{L:Reduce}
	Let $\mathcal{A}=\{a_1,\ldots, a_\ell\}$ be a nice family of polynomials with real exponents and $k$-parameters such that $\fdeg(a_1)>1$.
	Then there exists $a\in \mathcal{A}$ such that the family $\text{vdC}(\mathcal{A},a)$, ordered so that the first function is $\tilde{a}_1-a$, is nice and has smaller type.  Furthermore, if $\mathcal{A}$ consists of fractional polynomials with $k$-parameters, then $\text{vdC}(\mathcal{A},a)$ consists of fractional polynomials with $(k+1)$-parameters.
\end{lemma}
\begin{proof}
	We first remark that if $\mathcal{A}$ consists of fractional polynomials with $k$-parameters and $a$ is any fractional polynomial with $k$-parameters,  then \eqref{E:approx} implies that $\text{vdC}(\mathcal{A},a)$ consists of fractional polynomials with $(k+1)$-parameters.
	
	For $i=1,\ldots, \ell$, let $\tilde{a}_i$ be the polynomial with real exponents and $(k+1)$-parameters given by \eqref{E:approx}.
	We choose $a\in \mathcal{A}$ as follows:
	\begin{enumerate}
	 \item If  $a_1,\ldots, a_\ell$ do not have the same fractional degree, we let   $a_{i_0}$ be a function in the family $\{a_2,\ldots, a_\ell\}$ that has minimal (positive) fractional degree, and set $a=a_{i_0}$.

	\item If  $a_1,\ldots, a_\ell$  have the same fractional degree,  we let
	$i_0\in \{1,\ldots, \ell\}$ be so that  $\tilde{a}_1-a_{i_0}$ has maximal degree within the  family $\tilde{a}_1-a_1, \ldots, \tilde{a}_1-a_\ell$ and set  $a=a_{i_0}$.
	\end{enumerate}



 {\bf Claim 1.} {\em  The family $\text{vdC}(\mathcal{A},a)$ is nice.}

By construction, all functions in $\text{vdC}(\mathcal{A},a)$ are non-constant (we have removed constant functions).
We first show that independently of the choice of $a$,
the difference of $\tilde{a}_1-a$ with a function in $\text{vdC}(\mathcal{A},a)$
is always non-constant (in the variable $t$); in the process we also show that $\fdeg(\tilde{a}_1-a)>0$. Suppose that  such a difference has the form $\tilde{a}_1-a_i$ for some  $i\in \{1, \ldots, \ell\}$.
 It follows from Lemma~\ref{L:shifts}
 that $\tilde{a}_1$ contains the term $h_{k+1}a_1'(t)$, which depends non-trivially on  the parameter
 $h_{k+1}$ (note also that  $a_1,\ldots, a_\ell$ do not depend on this parameter). It follows from this
   and our assumption  $\fdeg(a_1)>1$ that
$$
\fdeg(\tilde{a_1}-a_i)\geq \fdeg(a_1')=\fdeg(a_1)-1>0, \quad i=1,\ldots, \ell.
$$
It remains to cover the case where  the difference of $\tilde{a}_1-a$ with a function in $\text{vdC}(\mathcal{A},a)$ has the form
$\tilde{a}_1-\tilde{a}_i$ for some  $i\in \{2, \ldots, \ell\}$. Then using  Lemma~\ref{L:shifts}  and our assumption that $\mathcal{A}$ is nice, we get
$$
\fdeg(\tilde{a}_1-\tilde{a}_i)\geq \fdeg(a_1-a_i)>0, \quad i=2,\ldots, \ell.
$$

Next we show that $\tilde{a}_1-a$ has maximal fractional degree within the family $\text{vdC}(\mathcal{A},a)$.
Suppose first that we are in Case $(i)$. Since $\fdeg(a_{i_0})<\fdeg(a_1)$,
 we have that  $\tilde{a}_1-a_{i_0}$ has the same fractional degree as
$a_1$, which by assumption has maximal fractional degree within the family $\{a_1,\ldots, a_\ell\}$. We deduce that $\tilde{a}_1-a_{i_0}$
 has maximal fractional degree within the family  $\text{vdC}(\mathcal{A},a)$.
 Suppose now that we are in Case $(ii)$ and
let $i\in \{1,\ldots, \ell\}$.  Since  $a_i-a_{i_0}=(a_i-\tilde{a}_1)+(\tilde{a}_1-a_{i_0})$ and by the choice of $i_0$ we have $\fdeg(\tilde{a}_1-a_{i_0})\geq \fdeg(\tilde{a}_1-a_i)$, we deduce
that
\begin{equation}\label{E:i-1}
\fdeg(\tilde{a}_1-a_{i_0})\geq \fdeg(a_i-a_{i_0}), \quad i=1, \ldots, \ell.
\end{equation}
 Moreover, note that   $\tilde{a}_i-a_{i_0}=(\tilde{a}_i-a_i)+(a_i-a_{i_0})$
and
\begin{equation}\label{E:i0}
\fdeg(\tilde{a}_1-a_{i_0})\geq \fdeg(\tilde{a}_1-a_1)=\fdeg(a_1)-1\geq\fdeg(a_i)-1=\fdeg(\tilde{a}_i-a_i),
\end{equation}
where the two identities follow from  Lemma~\ref{L:shifts}, and the first estimate follows from the choice of $i_0$ and the second since the family
$\mathcal{A}$ is nice. We deduce from \eqref{E:i-1}  and \eqref{E:i0} that
\begin{equation}\label{E:i0'}
\fdeg(\tilde{a}_1-a_{i_0})\geq \fdeg(\tilde{a}_i-a_{i_0}), \quad i=1,\ldots, \ell.
\end{equation}
Combining \eqref{E:i-1} and \eqref{E:i0'} we get that
$\tilde{a}_1-a_{i_0}$ has maximal fractional degree  within the family  $\text{vdC}(\mathcal{A},a)$.

\smallskip

{\bf Claim 2.} {\em  The family $\text{vdC}(\mathcal{A},a)$ has smaller type.}

Using Lemma~\ref{L:shifts} and the definition of the degree, it is easy to verify that if  for some $i\in \{1,\ldots, \ell\}$ we have $a_i\not\cong a_{i_0}$, then $\deg(a_i-a_{i_0})=\deg(\tilde{a}_i-a_{i_0})=\deg(a_i)$ and  $a_i-a_{i_0}\cong \tilde{a}_i-a_{i_0}$,
while if  $a_i\cong a_{i_0}$, then $\deg(a_i-a_{i_0})<\deg(a_i)$ and  $\deg(\tilde{a}_i-a_{i_0})<\deg(a_i)$.
Using these facts we easily get the following:

If we are  in Case $(i)$, we have that the
 type of $\mathcal{A}$ has the form
$(d, k_d, \ldots, k_l,0,\ldots, 0)$, where $l=\deg(a_{i_0})$,  $k_l\geq 1$, and $d\geq 1$.
 Then
  the type of
$\text{vdC}(\mathcal{A},a)$ is
$(d, k_d, \ldots, k_l-1)$  if $l=0$, and
$(d, k_d, \ldots, k_l-1,k_{l-1},\ldots, k_0)$
for some $k_0,\ldots, k_{l-1}\in \Z_+$  if $l\geq 1$.

If  we are in Case $(ii)$, we have that
 the type of $\mathcal{A}$ has the form
$(d, k_d,0, \ldots, 0)$, where  $d\geq 1$ and $k_d\geq 1$. Then for every $a\in \mathcal{A}$ the type of
$\text{vdC}(\mathcal{A},a)$ is $(d, k_d-1,k_{d-1}\ldots, k_0)$
for some $k_0,\ldots, k_{d-1}\in \Z_+$.

In both cases the type of
the family  $\text{vdC}(\mathcal{A},a)$
is smaller than the type of the family $\mathcal{A}$,
completing the proof of  Claim 2.
\end{proof}
\subsection{Proof of Theorem~\ref{T:CF'}} We will  now  use a  PET-induction technique   to prove Theorem~\ref{T:CF'}. The base case of the induction was covered in the previous section and the inductive step will be proved  using \eqref{E:VDC1}  and  Lemma~\ref{L:Reduce}.
\begin{proof}[Proof of Theorem~\ref{T:CF'}]
Our goal is to show that there exists $s\in \N$ such that if
 $f_{N,\uh,1}=f_1$, $\uh\in [L_N]^k,N\in\N$, and
$\nnorm{f_1}_s=0$, and all other functions below are assumed to be $1$-bounded, then
	$$
		\lim_{N\to\infty}	\E_{\underline{h}\in [L_N]^k}	
	\norm{\E_{n\in [N]}\,  w_{N,\underline{h}}(n)\cdot \prod_{i=1}^\ell T^{[a_i(\underline{h},n)]}f_{N,\uh,i}}_{L^2(\mu)}=0,
$$
	where  $w_{N,\uh}(n):=(\Delta_\uh\Lambda')(n)\cdot c_{N,\uh}(n)$, $ \uh\in [L_N]^k, n\in [N],  N\in\N,$ and  the sequence $(c_{N,\uh}(n))$ is $1$-bounded.
	
We prove this using induction on the type of the nice family of fractional polynomials  $\mathcal{A}:=\{a_1,\ldots, a_\ell\}$ with finitely many parameters. If $\fdeg(a_1)<1$ (then also $\fdeg(a_j)<1$ for $j=2,\ldots, \ell$), then the  result follows from Proposition~\ref{P:base}.

Suppose that the family
	 $\mathcal{A}:=\{a_1,\ldots, a_\ell\}$ has type $(d,k_d,\ldots, k_0)$,
	 where $d\geq 1$, $k_d\geq 1$,  $k_{d-1},\ldots, k_0\in \Z_+$, and the statement holds for all families
	 of fractional polynomials with finitely many parameters and type strictly smaller than
	 $(d,k_d,\ldots, k_0)$.  Since $\deg(a_1)\geq 1$ and $a_1$ is a fractional polynomial we have that $\fdeg(a_1)>1$.

	By Lemma~\ref{L:Reduce}, there exists $a\in \mathcal{A}$ such that the family  $\text{vdC}(\mathcal{A},a)$, ordered so that the first function is $\tilde{a}_1-a$ (where $\tilde{a}_1$ is as in \eqref{E:approx}), consists of fractional polynomials with finitely many parameters and satisfies the following
	\begin{equation}\label{E:choicea}
	\text{vdC}(\mathcal{A},a) \quad \text{is nice and has type strictly smaller than} \quad (d,k_d,\ldots, k_0).
	\end{equation}
	
We use  \eqref{E:VDC1} for the average $\E_{n\in [N]}$, compose with $T^{-[a(\uh,n)]}$
and then use  the Cauchy-Schwarz
inequality. We get that it suffices to show the following (recall that
$(\Delta_h u)(n)=u(n+h)\cdot \overline{u(n)}$)
	\begin{multline*}
\lim_{N\to\infty}	\E_{(\underline{h},h_{k+1})\in [L_N]^{k+1}}	
\Big|\Big|\E_{n\in [N]}\,  (\Delta_{h_{k+1}}w_{N,\underline{h}})(n)\cdot  \prod_{i=1}^\ell T^{[a_i(\underline{h},n+h_{k+1})]-[a(\uh,n)]}f_{N,\uh,i}\cdot \\ \prod_{i=1}^\ell T^{[a_i(\underline{h},n)]-[a(\uh,n)]}\overline{f}_{N,\uh,i}\Big|\Big|_{L^2(\mu)}=0.
\end{multline*}

We replace the differences of integer parts on the iterates  with the integer part of their differences and also replace $a_i(\uh,n+h_{k+1})$ with $\tilde{a}_i(\uh, h_{k+1}, n)$,
where $\tilde{a}_j$ is associated to $a_j$ by \eqref{E:approx} of Lemma~\ref{L:shifts}.
To make these substitutions we introduce some  error sequences that take finitely many values; as usual, these sequences can be   handled after we apply  Lemma~\ref{L:errors} (which applies without a problem since  the values of $n$ that are smaller than $\sqrt{N}$ contribute negligibly in the average).
After completing these maneuvers we  see that it suffices to show the following
\begin{equation*}
\lim_{N\to\infty}		\E_{(\uh,h_{k+1})\in [L_N]^{k+1}} \norm{\E_{n\in [N]}\,  w_{N,\uh,h_{k+1}}(n) \cdot \prod_{i=1}^{2\ell} T^{[b_i(\uh,h_{k+1},n)]+\epsilon_{i,N}}g_{N,\uh,h_{k+1},i}}_{L^2(\mu)}=0,
\end{equation*}
where $\epsilon_{1,N},\ldots, \epsilon_{2\ell,N}$ take finitely many values for $N\in\N$,
$$
 w_N(\uh, h_{k+1},n):=
 (\Delta_{(\uh,h_{k+1})}\Lambda')(n)\cdot c_{N,\uh, h_{k+1}}(n),
$$
for some $1$-bounded sequence  $(c_{N,\uh, h_{k+1}}(n))$, and
\begin{align*}
	b_i(\uh, h_{k+1},t)&:=\tilde{a}_i(\uh, t+h_{k+1})-a(\uh,t),\quad    i=1,\ldots, \ell,\\
	     b_{\ell+i}(\uh, h_{k+1},t)&:=a_i(\uh, t)-a(\uh,t), \quad   i=1,\ldots, \ell,
\end{align*}
and
$g_{N,\uh,h_{k+1},i}$ are $1$-bounded functions in $L^\infty(\mu)$ such that $g_{N,\uh,h_{k+1},1}:=f_1$ for all $(\uh,h_{k+1})\in [L_N]^{k+1}$, $N\in\N$. We compose with $T^{-\epsilon_{1,N}}$ inside the $L^2(\mu)$-norm and set $h_{N,\uh,h_{k+1},i}:=T^{\epsilon_{i,N}-\epsilon_{1,N}}g_{N,\uh,h_{k+1},i}$, $i=1,\ldots, 2\ell$ (then $h_{N,\uh,h_{k+1},1}=f_1$). We get that it suffices to show that
 \begin{equation}\label{E:indstep}
\lim_{N\to\infty}		\E_{(\uh,h_{k+1})\in [L_N]^{k+1}} \norm{\E_{n\in [N]}\,  w_{N,\uh,h_{k+1}}(n) \cdot \prod_{i=1}^{2\ell} T^{[b_i(\uh,h_{k+1},n)]}h_{N,\uh,h_{k+1},i}}_{L^2(\mu)}=0.
\end{equation}

Finally, we can remove all functions associated with iterates that do not depend on the variable $n$ (note that by Lemma~\ref{L:Reduce} the function $b_1$ is not one of them), and thus we arrive at an average with iterates given by the family
$\text{vdC}(\mathcal{A},a)$, ordered so that the first function is $\tilde{a}_1-a$. By the choice of $a$ we have that \eqref{E:choicea} holds. Hence,  the induction hypothesis applies for this family and gives that there exists $s\in \N$ such that if $\nnorm{f_1}_s=0$, then   \eqref{E:indstep} holds. This completes the induction step and the proof.
\end{proof}



\begin{thebibliography}{9999}

\bibitem{B87} V.~Bergelson. Weakly mixing PET. {\em Ergodic Theory
	Dynam. Systems} \textbf{7} (1987), no. 3, 337--349.




\bibitem{BKMST14} V.~Bergelson, G.~Kolesnik,
M.~Madritsch, Y.~Son, R.~Tichy.
Uniform distribution of prime powers and sets of recurrence and van der Corput sets in $\Z^k$.
{\em Israel J. Math.} {\bf 201} (2014), 729--760.



\bibitem{BKS19} V.~Bergelson, G.~Kolesnik,
Y.~Son.
Uniform distribution of subpolynomial functions along primes  and applications.
{\em J. Analyse Math.}  \textbf{137} (2019), 135--187.


  \bibitem{BMR21} V.~Bergelson, J.~Moreira,
F.~Richter.    Multiple ergodic averages along functions
from a Hardy field: convergence, recurrence and combinatorial applications. Preprint 2020, \texttt{arXiv:2006.03558}

 \bibitem{BM21} A.~ Best, A.~Ferr\'e Moragues.
Polynomial ergodic averages for certain countable ring actions.  To appear in {\em Discrete Contin. Dyn. Syst.}, \texttt{arXiv:2105.04008}

\bibitem{F10} N.~Frantzikinakis. Multiple recurrence and convergence for Hardy sequences of polynomial
growth.  {\em J. Analyse Math.}  \textbf{112} (2010), 79--135.

\bibitem{Fr16} N.~Frantzikinakis. Some open problems on multiple ergodic averages. {\em Bull. Hellenic Math. Soc.} \textbf{60}
(2016),  41--90.

\bibitem{F21} N~Frantzikinakis. Joint ergodicity of sequences. Preprint 2021, \texttt{arXiv:2102.09967}

\bibitem{FHK11} N.~Frantzikinakis, B.~Host, B.~Kra. The polynomial multidimensional Szemer\'edi Theorem along shifted primes.
{\em Isr. J. Math.} {\bf 194} (2013), 331--348.


\bibitem{Fu77} H.~Furstenberg.
Ergodic behavior of diagonal measures and a theorem of Szemer\'edi
on arithmetic progressions. {\em J. Analyse Math.} \textbf{31}
(1977), 204--256.



\bibitem{Fu81a} H.~Furstenberg.
{\em Recurrence in ergodic theory and combinatorial number theory}.
	Princeton University Press, Princeton, 1981.

\bibitem{G76} P.~Gallagher. On the distribution of primes in short intervals. {\em Mathematika} \textbf{23} (1976), 4--9.

\bibitem{GT08} B.~Green, T.~Tao.
The primes contain arbitrarily long arithmetic progressions.  {\em Annals Math.}
\textbf{167} (2008), 481--547.

\bibitem{GT09} B.~Green, T.~Tao. Linear equations in primes.
{\em  Annals  Math.} \textbf{171} (2010), 1753--1850.

\bibitem{HR74} H.~Halberstam, H.~Richert.  {\em Sieve methods},  Academic Press, New York, 1974.

\bibitem{HK05} B.~Host,  B.~Kra.
Non-conventional ergodic averages and nilmanifolds. {\em Ann. of Math.}  \textbf{161}  (2005), 397--488.


\bibitem{HK18}
B.~Host, B.~Kra.
{\em Nilpotent structures in ergodic theory.}
Mathematical Surveys and Monographs, vol. 236. American Mathematical Society, Providence, RI, 2018.

\bibitem{IK04}
H.~Iwaniec, E.~Kowalski. {\em  Analytic number theory.} American Mathematical
Society Colloquium Publications, vol. 53. American Mathematical Society, Providence,
RI, 2004.

\bibitem{KK19} D.~Karageorgos, A.~Koutsogiannis. Integer part independent polynomial averages and applications
along primes. {\em Studia Mathematica} \textbf{249} (2019), 233--257.


\bibitem{K15} A.~Koutsogiannis.
Closest integer polynomial multiple recurrence along shifted primes.
{\em Ergodic Theory Dynam.
	Systems} \textbf{38}  (2018), no. 2, 666--685 .


\bibitem{K11}
E.~Kowalski. Averages of Euler products, distribution of singular series and the ubiquity of Poisson
distribution. {\em Acta Arith.} \textbf{148}  (2011), 153--187.


\bibitem{KN74} L.~Kuipers, H.~Niederreiter.  {\em Uniform distribution
of sequences.}  Pure and Applied Mathematics.
	Wiley-Interscience, New York-London-Sydney, 1974.

\bibitem{L05} A.~Leibman.  Convergence of multiple ergodic
averages along polynomials of several variables.  {\em Isr. J.
	Math.} \textbf{146} (2005), 303--316.




\bibitem{P19} S.~Peluse. On the polynomial Szemer\'edi theorem in finite fields. {\em Duke
	Math. J.} \textbf{168} (2019), 749--774.


\bibitem{PP19} S.~Peluse, S.~Prendiville. Quantitative bounds in the non-linear Roth Theorem. To appear in {\em Int. Math. Res. Not. (IMRN)},   \texttt{arXiv:1903.02592}

\bibitem{R21} F.~Richter.
Uniform distribution in nilmanifolds along functions from a Hardy field.
To appear in  {\em J. Analyse Math.}, \texttt{arXiv:2006.02028}



\bibitem{TZ08} T.~Tao, T.~Ziegler.
The primes contain arbitrarily long polynomial progressions.
{\em Acta Math.} \textbf{201} (2008), 213--305.


\bibitem{TZ15} T.~Tao, T.~Ziegler. Narrow progressions in the primes.  {\em  Analytic Number Theory.}
Springer International Publishing (2015), 357--379.


\bibitem{WZ11} T.~Wooley, T.~Ziegler.  Multiple recurrence and
convergence along the primes.  {\em Amer. J. of Math.}
{\bf 134} (2012), 1705--1732.

\end{thebibliography}
\end{document}